\documentclass[10pt]{amsart}
\usepackage{a4}
\usepackage{amsmath,amssymb,amsfonts,amsthm}
\usepackage{color}

\usepackage{ifpdf}
\ifpdf
    \usepackage[pdftex]{graphics,graphicx}
    \DeclareGraphicsExtensions{.png,.pdf}
\else
   \usepackage[dvips]{graphics,graphicx}
   \DeclareGraphicsExtensions{.eps}
\fi
\graphicspath{{.}{figures/}}

\theoremstyle{plain}

\newtheorem{theorem}{Theorem}[section]
\newtheorem{lemma}[theorem]{Lemma}
\newtheorem{proposition}[theorem]{Proposition}

\newtheorem{remark}[theorem]{Remark}
\newtheorem{definition}[theorem]{Definition}
\theoremstyle{definition}
\numberwithin{equation}{section}

\def\ds{\displaystyle}
\def\Om{\Omega}
\def\R{\mathbb{R}}
\def\Z{\mathbb{Z}}
\def\N{\mathbb{N}}
\def\dist{\textup{dist}}
\def\osc{\textup{osc}}
\def\H{\mathcal{H}}
\def\P{Per}
\def\E{\mathcal{E}}
\def\F{\mathcal{F}}
\def\M{\mathcal{M}}
\def\RT{\mathcal{R}_T}
\newcommand{\Symd}{\mathbb M^{d\times d}_{\rm sym}}
\newcommand{\Clo}{\mathcal C(\R^d)}
\newcommand{\Ope}{\mathcal A(\R^d)}
\newcommand{\f}{\varphi}
\newcommand{\e}{\varepsilon}
\newcommand{\hs}{\mathcal H}
\newcommand{\pro}{\mathcal P_{\hat p}}

\title
[A Non-Local Mean Curvature Flow]
{A Non-Local Mean Curvature Flow and its semi-implicit
time-discrete approximation}
\author[A. Chambolle]
{Antonin Chambolle}
\address[Antonin Chambolle]{CMAP, Ecole Polytechnique, CNRS, France}
\email[A. Chambolle]{antonin.chambolle@cmap.polytechnique.fr}

\author[M. Morini]
{Massimiliano Morini}
\address[Massimiliano Morini]{Dip.~di Matematica, Univ.~Parma, Italy}
\email[M. Morini]{massimiliano.morini@unipr.it}

\author[M. Ponsiglione]
{Marcello Ponsiglione}
\address[Marcello Ponsiglione]{Dip.~di Matematica,
Univ.~Roma-I ``La Sapienza'', Roma, Italy}
\email[M. Ponsiglione]{ponsigli@mat.uniroma1.it}

\keywords{nonlocal curvature flows, nonlocal geometric flows,
minimizing movements, viscosity solutions}

\begin{document}
\bibliographystyle{plain}
%\vskip .2truecm
\begin{abstract}
We address in this paper the study of a geometric evolution,  corresponding to a  curvature which  is   non-local and singular at the origin.
The curvature represents the first variation of the  energy $\M_\rho(E)$ defined in \eqref{defmink0},
%$$
%\M_\rho(E)\ =\ \frac{1}{2\rho} \left|\left\{x\in \R^d\,:\, \dist(x,E)\le \rho\,,
%\dist(x,\R^d \setminus E)\le\rho\right\}\right|,
%$$
proposed in a recent work~\cite{BKLMP} as a variant of the standard perimeter  penalization   for the
denoising of nonsmooth curves. 
\par
To deal with such degeneracies, we first give  an abstract 
%setting within the  level set aproach, and provide 
existence and uniqueness result for  viscosity solutions of
non-local degenerate Hamiltonians, satisfying suitable continuity assumption 
with respect to Kuratowsky convergence of the level sets.  
This abstract setting applies to an  approximated flow. Then,  by the method
of minimizing movements,  we also build an ``exact'' curvature flow, and we illustrate  some examples,  comparing the results with the standard mean curvature flow.  
\end{abstract}
%%%% was:
%We address in this paper the study of a non-local ``curvature flow''
%corresponding to  the first variation of the Minkowski content. Such curvature   is   non-local and degenerate at the origin.
%%\par
%To deal with such degeneracies, we first use   the viscosity machinery, and provide existence and uniqueness for
%non-local degenerate Hamiltonians, satisfying suitable continuity assumption 
%with respect to Kuratowsky convergence of the sets.  
%This abstract setting applies to an  approximate flow. Then,  by the method
%of minimizing movements,  we also build, an ``exact'' curvature flow. 
%\par
%Finally, we mention that the Minkowski content has been  proposed as a variant of the standard perimeter  penalization,   in a recent work~\cite{BKLMP} for the
%denoising of nonsmooth curves. 
%We illustrate with some examples that also the non-local geometric flow presents specific features, that we compare with the standard mean curvature flow.  
%
%\vskip .3truecm
%\noindent Keywords: Mathematics
%
%\vskip.1truecm
%\noindent 2000 Mathematics Subject Classification:}
%\keywords{blah}
%%%%%%%%%%%%%%
\maketitle

%{\small \tableofcontents}
%%%%\problem{Introduction?}

\section{Introduction}

In a recent paper \cite{BKLMP}, the last two authors, together with
M.~Barchiesi, S.~H.~Kang and T.~Le, proposed a variational model
for (binary) image denoising, which was supposed to preserve small
scale details (or small oscillations of the boundary) while regularizing
the large scales. This model is a variant of the celebrated Mumford-Shah functional, 
where the perimeter term is replaced by  the following one
\begin{equation}\label{defmink0}
\M_\rho(E)\ =\ \frac{1}{2\rho} \left|\left\{x\in \R^d\,:\, \dist(x,E)\le \rho\,,
\dist(x,\R^d \setminus E)\le\rho\right\}\right|,
\end{equation}
defined for any $E\subseteq\R^d$, where $\rho>0$ acts as a scale selection
parameter. Here and throughout the paper, given $A\subset \R^d$ measurable, we denote by  $|A|$ its Lebesgue measure.  Notice that the energy is finite if and only if  $\partial E$ is compact.
The idea behind such a variant is that fluctuations of $\partial E$ at lengths
much smaller than $\rho>0$ will have very little influence on
the energy; on the other hand, it  behaves as the standard perimeter on much larger smooth boundaries,
and in a more complicated non-local way  on sets with fine microstructures on scales of order $\rho$. 

\begin{figure}[htb]
\centerline{
\includegraphics[width=3cm]{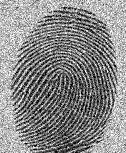}
\hspace{1cm} \includegraphics[width=3cm]{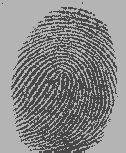}
}
\caption{An example from~\cite[Fig.~4.1]{BKLMP}: the fingerprint (left:
noisy, right: denoised).}
\label{fig:fing}
\end{figure}
%The idea was to consider a functional which is, at large scale
%and for smooth sets, an approximation of the perimeter, while
%it becomes almost blind to small oscillations of the boundaries or small scales.
\noindent
The sort of denoising which is obtained in~\cite{BKLMP} is
shown in Fig.~\ref{fig:fing}, where small oscillations (here
the stripes of the fingerprints) are almost untouched, while the noise
has mostly been removed.

In this paper, we try to investigate some mathematical analysis aspects of
this model.  More precisely, we want to study the geometric
evolution of curves and shapes by the gradient flow of the functional
proposed by these authors. 

To this purpose, we first extend our energy 
to $L^1$ functions, and express it in terms of a function depending on the oscillation of $u$ on balls of radius $\rho$, following the approach in \cite{CDarbon}. With this point of view, it turns out that  \eqref{defmink0}   is the restriction to characteristics functions  of a convex, l.s.c.~functional, satisfying a suitable  ``coarea formula''. 
Then, 
we introduce the ``curvature'' as the first variation
of this functional with respect to inner variations of the sets.
This curvature is not continuous and it is not  well defined for all
smooth sets. Therefore, in \eqref{varmr} we introduce a smoother version $\M^f$
of~\eqref{defmink0}, which roughly speaking consists in averaging $\M_r$ over  $r$, for  $r$ varying in a neighborhood of $\rho$.  The corresponding  curvature  is now well defined on smooth sets. 
%Nevertheless, such non-local curvature is singular at the origin. 

After this preliminary analysis to define a proper notion of curvature,  we study the corresponding geometric flow. { Using a level set approach and} working in the framework of viscosity solutions, we define a mean curvature flow equation, which is
both non-local and singular. 
Indeed, our Hamiltonian $F(x, D u,D^2 u , K)$
depends in a non-local way on the level set $K$, and
behaves like a power ($d-1$) of the curvature tensor of { $\partial K$}
for vanishing sets, being thus singular in  dimension $d\ge 3$ {(see \eqref{cippi})}. 
To deal with such degeneracy we combine 
the  approach by Slep\v cev~\cite{S} to non-local Hamiltonians  with the approach by Ishii and Souganidis~\cite{IS} and
Goto~\cite{Goto}  to degenerate Hamiltonians. 
However, the  approach in \cite{S} is based on the assumption that the  Hamiltonian is continuous with respect to all its variables, 
in particular with respect to $L^1$ convergence of the sets. This is not the case of our Hamiltonian (and of any reasonable regularization of it).
Therefore, we build up a variant of the approach in \cite{S} that works for a general class of Hamiltonians satisfying suitable  continuity properties  with respect to the
Kuratowski convergence instead of $L^1$ convergence of sets.
We adapt the  notion in \cite{IS} of  viscosity
solutions for singular Hamiltonians to our non local setting, and we show a corresponding result of existence
and uniqueness. This result will apply to a general class of Hamiltonians, which does not include the Hamiltonian corresponding to our non-local curvature flow,  but only a {suitable} continuous approximation of it.

Finally, we study  the minimizing movements corresponding to the energy $\M^f$. We introduce  an implicit time-discretization of
the motion, and we  show that it converges, up to a subsequence,
to a solution of the level set equation in the viscosity sense. 
In this way we recover an existence result for viscosity solutions
also for the exact Hamiltonian corresponding to the first variation
of $\M^f$, yet without uniqueness. 
We mention that in a recent paper of Caffarelli and
Souganidis~\cite{CaffaSouga2010},
a similar strategy (based, this time, on a diffusion/thresholding 
time-discrete scheme) has been successfully implemented to build up
a non-local curvature flows associated to fractional diffusions.
Our time-discretization approach can be numerically implemented,
following the approach in \cite{CDarbon}: we show eventually in
Section~\ref{secnumer} a few examples which are
compared with the classical mean curvature flow, and seem to confirm a
slower smoothing of oscillatory boundaries. 

We mention the existence of a few interesting alternative approaches
to non-local evolutions. The recent papers of~\cite{ImbertNonLocal,BaLeMi11}
provide a point of view slightly different from ours, and address
different kinds of evolutions. In particular, \cite{BaLeMi11}
also deals with non-monotone evolutions, such as the one describing
the motion of dislocation lines in crystals (see
also~\cite{ACM05,APLBM,BCLM08,BCLM09,CDLFM07,IM10}).  Another approach
is described in the papers of Cardaliaguet~\cite{C-NLI,C-NLII},
Cardaliaguet and Rouy~\cite{CardaliaguetRouy}, Cardaliaguet and
Ley~\cite{CL07,CL08}.
There, appropriate definition for evolving tubes are proposed and
the convergence of a time-discrete scheme (of the same kind as ours)
is addressed in~\cite{CL08}.
Moreover, except in the preliminary work~\cite{C-NLI}, the
authors of this series of papers
have taken care to never need to evaluate the
velocity on arguments which are not ``natural'' (such as smooth level sets
and their normal or second fundamental form), contrarily to what
is needed in our proof of uniqueness (as in~\cite{S}).
Unfortunately, extending their work to our approach raises complicated
technical issues, since in particular our velocity does not have
the required continuity properties, and our minimizers have unknown
regularity. This is an interesting direction for future research,
but we also believe it is useful to develop the level-set
approach in the non-local geometric setting.
\medskip

To summarize, the first goal of this paper is to investigate the geometric flow corresponding to a non-local variant of the perimeter introduced in 
\cite{BKLMP}, in connection with   image denoising.
We have  
developed a viscosity approach to non-local singular Hamiltonians, combining many ideas from
\cite{S}, \cite{IS} and \cite{Goto}. 
%We have obtained existence and uniqueness of viscosity solutions for a large class of such non-local singular Hamiltonians,
Through the viscosity approach we have obtained existence and uniqueness  for a suitable regularization of our  Hamiltonian, while  
a minimizing movements approach yields
 a  solution for the original Hamiltonian. The abstract approach is itself interesting and stimulating: a complete picture  at the moment is still missing, 
and  this paper represents  a first attempt to study  singular non-local Hamiltonians,
not even continuous with respect to $L^1$ convergence, but only   with respect to Kuratowski convergence of level sets.  { The combination of the
minimizing movements variational method  with  viscosity techniques} seems to be a promising approach to complete the picture. To our knowledge, up to now this
kind of study { has been carried out} only in~\cite{CL08}, \cite{CiomagaThouroude},
and a few papers by the first author and co-workers. We hope that
(borrowing in particular from~\cite{CL08}) we will be able to extend
these ideas to other motions, and also, understand how to make the proof of
the comparison result less dependent on the extension of the Hamiltonian
out of its natural domain of definition.

\section{The energy functionals}

\subsection{The $\rho$-neighborhood}
As mentioned, we focus on the study of \eqref{defmink0}.
It is well-known  that, under mild regularity assumption on  $E$ (see
for instance \cite{AmFuPa:00}) we have
%a set with finite perimeter such
%that $\H^{d-1}(\Om\cap({\partial E}\setminus \partial^* E))=0$,
\[
\lim_{\rho\to 0} \M_\rho(E)\ =\ \H^{d-1}(\partial E)
\ =\ \P(E),
\]
where $\P(E)$ is the standard perimeter of $E$.
It is also very easy to show that $\M_\rho$ $\Gamma$-converges to the standard perimeter~\cite{CLL-201x}.
%, then $\E_\rho(E)$    is  an approximation of the standard perimeter $\P(E)$:
An issue with definition~\eqref{defmink0} is that it depends on
the choice of the representative  within the Lebesgue equivalence class of the set $E$. 
%, with respect to the equivalent relation $E\sim E'$ if $|E\triangle E'|=0$, where $E \triangle E' = (E\setminus E') \cup (E'\setminus E)$. 
For this reason, one
may introduce the following variant:
\begin{equation}\label{defmink}
\E_\rho(E)\ =\ \frac{1}{2\rho}\int_{\R^d} \osc_{B(x,\rho)}(\chi_E)\,dx
\end{equation}
where $\osc_A(u)$ denotes the \textit{essential oscillation} of the
measurable function $u$ over a measurable set $A$, defined by
\[
\osc_A(u)\ =\ \textup{ess} \sup_A u\,-\, \textup{ess} \inf_A u.
\]
One checks that $\E_\rho(E)$ coincides with the measure of the $\rho$-neighborhood
of the  essential boundary of $E$. Moreover, 
\[
\E_\rho(E)\ =\ \inf \{ \M_\rho(E'): \,  |E\triangle E'| = 0 \}, 
\]
where $E\triangle E'$ denotes the symmetric difference $(E\setminus E') \cup (E'\setminus E)$.
Finally, $\E_\rho(E) = \E_\rho(E^c)$ and it is finite if and only if either $E$ or $E^c$ is (essentially) bounded (where $E^c:=\R^d\setminus E$).   
An advantage of Definition~\eqref{defmink} is that can be
easily generalized to a measurable function $u\in L^1_{loc}(\R^d)$.
By a slight abuse of notation, we still denote $\E_\rho(u)$ the functional:
\begin{equation}\label{defminku}
\E_\rho(u)\ =\ \frac{1}{2\rho}\int_{\R^d} \osc_{B(x,\rho)}(u)\,dx.
\end{equation}
One can check that this energy is one-homogeneous, convex and therefore sub-additive. Moreover,  it is lower semicontinuous with respect  to weak$^*$ convergence in  $L^\infty_{loc}$, 
and  satisfies the following  generalized coarea formula
\begin{equation}\label{coarea}
\E_\rho(u)\ =\ \int_{-\infty}^\infty \E_\rho(\{u>s\})\,ds.
\end{equation}
This follows from the fact that for any $A$ and $u$ we have
\[
\osc_A (u) \ =\ \int_{-\infty}^\infty\osc_A(\chi_{\{u>s\}})\,ds.
\]
Moreover, one easily deduces that, given two measurable sets $A,B\subset \R^d$,
\begin{equation}\label{submod}
\E_\rho(A\cup B)\ +\ \E_\rho(A\cap B)\ \le \ \E_\rho(A)\ +\ \E_\rho(B)\,.
\end{equation}
Indeed,   observing that  $\chi_{A\cup B}+\chi_{A\cap B}=\chi_{A}+\chi_B$, by coarea formula  and  in view of the subadditivity of $\E_\rho$, we have 
\begin{align}
\nonumber \E_\rho(A\cup B) + \E_\rho(A\cap B) &= \int_{-\infty}^\infty \E_\rho(\{\chi_{A\cup B} + \chi_{A\cap B} > s \}) \, ds \\&=
  \E_{\rho}(\chi_A + \chi_B)\le
   \E_{\rho}(\chi_A ) +  \E_{\rho}( \chi_B) = \E_\rho(A) + \E_\rho(B).
\end{align}

\subsection{The continuous energy functional}
Fix  $\rho_0>0$ and $\underline\delta\in(0, \rho_0)$.  We consider now a Lipschitz  function $f:\R\to \R_+$ which is even, with $\textup{supp} f=[-\rho_0,\rho_0]$,
constant in $[-\underline{\delta},\underline{\delta}]$ and
nonincreasing in  $\R_+$.

We then introduce the following variant of~\eqref{defmink0}:
\begin{equation}\label{varmr}
\M^f(E)\ =\ \int_{\R^d} f(d_E(x))\,dx,
\end{equation}
where  $d_E$ is the signed distance to $\partial E$ (negative
inside $E$ and positive outside). 
Notice that $\M^f(E)$ is finite if and only if $\partial E$ is compact, i.e., $E$ or its complement is bounded. 
We now show that
$$
\M^f(E)= \int_0^{\rho_0} (-2s f'(s)) \M_s(E)\,ds.
$$
Since $\M^f(E) = \M^f( E^c)$ we can assume that 
$E^c$ is bounded. 
Then,
thanks to the co-area formula we have 
\begin{align*}
\M^f(E)\ 
&=\ \int_{\R^d} f(d_E(x))|D d_E(x)|\,dx 
= \int_{-\rho_0}^{\rho_0} f(s)\H^{d-1}(\partial \{d_E>s\} )\,ds
\\
 &= { - \int_{-\rho_0}^{\rho_0}  - f'(s) |\{s<d_E<\rho_0\} |\,ds }\\
& = { \int_{0}^{\rho_0}  - f'(s) (  |\{- s<d_E<\rho_0\}| - |\{s<d_E<\rho_0\}|) \,ds  }
\\ 
&=   \int_0^{\rho_0} (-2s f'(s)) \M_s(E)\,ds.
\end{align*}
Thanks to this, we can introduce the following variant of $\M^f$,
defined on Borel sets and which depends only on the Lebesgue equivalence
class 
\begin{equation}\label{econeffe}
\E^f(E)\ =\ \int_0^{\rho_0} (-2s f'(s)) \E_s(E) \,ds.
\end{equation}
As before, we consider the convex extension of $\E^f$, defined for all functions $u\in L^1_{loc}(\R^d)$ by  
\begin{equation}\label{econeffu}
\E^f(u)\ =\ \int_0^{\rho_0} (-2s f'(s)) \E_s(u) \,ds=\ \int_\Om 
\int_0^{\rho_0}  (-2s f'(s)) \osc_{B(x,s)}(u)\,ds\,dx.
\end{equation}
\noindent
By construction, $\E^f$ is a convex, lower semicontinuous energy
which satisfies the generalized coarea formula
\begin{equation}\label{coareaEf}
\E^f(u)\ =\ \int_{-\infty}^\infty \E^f(\{u>s\})\,ds.
\end{equation}
Clearly, \eqref{submod}  is still  true also for $\E^f$.

\subsection{The non-local curvature}\label{seccurv}

Let $E\subset\R^d$ be a smooth  set with compact boundary. We denote by 
$\nu_E(x)$ the outer normal unit vector to $\partial E$ at $x$.
The non-local curvature $\kappa_\rho$ is formally defined as the first variation of the energy $\E_\rho$ in \eqref{defmink}. 
Set
\begin{equation}
\kappa_\rho(E,x)\ =\ \kappa^+_\rho(E,x)\,+\,\kappa^-_\rho(E,x),
\end{equation}
where
\begin{equation}\label{kappaprho}
\kappa^+_\rho(E,x)\ =\begin{cases}
\ds \frac{1}{2\rho}\det (I+\rho \nabla \nu_E(x))\  &
\textrm{ if }  \dist(x+\rho \nu_E(x),E)= \rho\,, \\[2mm]\\
0 & \textrm{ otherwise,}
\end{cases}
\end{equation}
\begin{equation}
\kappa^-_\rho(E,x)\ =\begin{cases}
\ds - \frac{1}{2\rho}\det (I-\rho \nabla \nu_E(x)) &
\textrm{ if }  \dist(x -\rho \nu_E(x),E^c)= \rho\,, \\[2mm]\\
0 & \textrm{ otherwise.}
\end{cases}
\end{equation}
Let us decompose $\partial E$ into three sets: $\partial E= A_\rho^+ \cup B_\rho^+ \cup  \mathcal N_\rho^+$, where  
\begin{itemize}
\item[] $A_\rho^+:=\left\{ x \in \partial E: \text{ there exists } t > \rho: \,   \dist(x+t \nu_E(x),E)= t \right\}$, 
\item[] $B_\rho^+:=\left\{ x \in \partial E: \dist(x+\rho \nu_E(x),E)<\rho \right\}$.
\item[] $\mathcal N_\rho^+ = \partial E\setminus (A^+_\rho \cup B^+_\rho)$.
\end{itemize}
Analogously, we define$A_\rho^-, \, B_\rho^- ,\,  \mathcal N_\rho^-$ with $\nu_E(x)$ replaced by $-\nu_E(x)$, and we set $\mathcal N_\rho: = \mathcal N_\rho^+\cup \mathcal N_\rho^-$.
\begin{lemma}\label{curvatura}
Let $E\subset \R^n$ be a set of class $C^2$ with compact boundary, such that  $\hs^{d-1}(\mathcal N_\rho)=0$.
Then, for every $\f\in C^2 (\partial E;\R)$  we have 
\begin{equation}\label{first variation}
\frac{d}{d \e} \E_\rho\big(\Phi_\e(E)\big)_{|_{\e=0}} = \int_{\partial E} \kappa_\rho \f \, d \hs^{d-1},
\end{equation}
where $\Phi_\e$ is a diffeomorphism such that  $\Phi(x)= x + \e\f(x) \nu_E(x)$ for $x\in\partial E$. 
\end{lemma}

\begin{remark}\label{nonlocvel}
{\rm
Notice that the assumption of Lemma \ref{curvatura} holds true for generic smooth sets. More precisely, given a smooth set $E$, then for almost all positive $\rho$ one has $\hs^{d-1}(\mathcal N_\rho) = 0$.
Moreover, such assumption is crucial. Indeed, let $E$ be  a rectangle of sides $2$ and $4$, respectively, and set $\rho = 1$. In this case, a curvature $\kappa_\rho$ satisfying \eqref{first variation} is not well defined. Indeed, one readily sees that such curvature $\kappa_\rho$ should depend in a non-local way on $\f$. More precisely, let $\f = \f_1 + \f_2$, where $\f_i$ are defined in a small neighborhood of the middle points $p_i$ of the large sides $L_i$ of $E$, and assume that $\f_i$ have constant sign. Then, if $\f_1 (p_1) + \f_2(p_2)>0$, then \eqref{first variation}  holds true, while if $\f_1 (p_1) + \f_2(p_2)<0$, then \eqref{first variation} holds true with $\kappa_\rho$ replaced by $\kappa^+_\rho$. In particular, the ``curvature'' at $p_1$ depends on the value of $\f$ at $p_2$.
}
\end{remark}

\begin{proof}[Proof of Lemma \ref{curvatura}]
We will show that 
\begin{equation}\label{first variationpm}
\frac{d}{d \e} \E^\pm_\rho\big(\Phi_\e(E){| E}\big)_{|_{\e=0}} = \int_{\partial E} \kappa^\pm_\rho \f \, d \hs^{d-1},
\end{equation}
where, for every set $F$ 
$$
\E^+_\rho(F|E) :=  \frac{1}{2\rho}\int_{E^c} \osc_{B(x,\rho)}(\chi_F)\,dx \quad \text{ and } \quad
\E^-_\rho(F|E) =  \frac{1}{2\rho}\int_{ E} \osc_{B(x,\rho)}(\chi_F)\,dx. 
$$

In order to prove \eqref{first variation}, we will focus on the identity
\begin{equation}\label{first variationp}
\frac{d}{d \e} \E^+_\rho\big(\Phi_\e(E)\big)_{|_{\e=0}} = \int_{\partial E} \kappa^+_\rho \f \, d \hs^{d-1},
\end{equation}
the variation of $\E^-_\rho$ being analogous.  The proof is divided in three steps: first, we prove that \eqref{first variationp} holds if the support of $\f$ is contained in  $A_\rho^+$. Then, we show that the variation vanishes on  $B_\rho^+$. Finally, by a localization argument, recalling also  $\hs^{d-1}( N_\rho^+)=0$, we deduce the validity of ~\eqref{first variationp}.\\
\smallskip

\noindent\textbf{Step 1.}
In this Step we assume that supp$(\f)\subseteq A_\rho^+$, and then prove \eqref{first variationp}.
For every $x\in A_\rho^+$ let $y_\e(x):= \Phi_\e(x)=x + \e \f(x)\nu_E(x)$, and set
%\begin{eqnarray}\label{defEcurv}
%E^+_{A_\rho^+,\e}: = \{ x + \delta \f(x)\nu_E(x), \, x\in A_\rho^+, \, \delta\in (0,\e)\} 
%\\ 
%\nonumber
%\bigcup 
%\{ x + \e \f(x)\nu_E(x) + t \nu_E(x + \e \nu_E(x)),\, x\in A_\rho^+, \, t\in (0,\rho) \}.
%\end{eqnarray}
\begin{equation}\label{defEcurv}
E^+_{\f,\e}: =
\left\{y_\e(x)+ t \nu_{\Phi_\e(E)}(y_\e(x)),\, x\in \partial E, 
\, t\in (-\rho,\rho) \right\} \setminus E,
\end{equation}
%Since $x\in A_\rho^+$, the (standard) curvature of $\partial E$ at $x$ is lower than $1/\rho$. It follows that, for $\e$ small enough  
so that, { for $\e$ small enough} 
$$
\E^+_\rho\big(\Phi_\e(E){| E}\big) = |E^+_{\f,\e}|. 
$$
Set now 
$$
\tilde E^+_{\f,\e} := \{ x + t \nu_E(x), x\in \partial E, \, t \in (0,\rho+\e\f(x)) \} \setminus E. 
$$
%Finally, set 
%\begin{equation}\label{Ge}
%G^+_{A_\rho^+,\e} := \tilde E^+_{A_\rho^+,\e}  \setminus E^+_{A_\rho^+, 0} = \{ x + \delta \nu_E(x), x\in A_\rho^+, \, \delta \in (\rho,\rho+\e\f(x)) \}. 
%\end{equation}
By construction, since $\partial E$ is of class $C^2$,  one can see that $| E^+_{\f,\e} \triangle \tilde  E^+_{\f,\e}|= o(\e)$. %where $o(\e)/\e\to 0$ as $\e\to 0$.
Finally, we have
%\begin{equation}
%\begin{split}
\begin{multline*}
\frac{\E^+_\rho\big(\Phi_\e(E){| E}\big) - \E^+_\rho\big(E{| E}\big) }{\e} \ 
=\ \frac{|E^+_{\f,\e} | - |E^+_{\f,0} |}{\e}\  =\ 
\frac{|\tilde E^+_{\f,\e}| - |E^+_{\f,0} |}{\e}  + o(1) \\ %\frac{o(\e)}{\e}\\ 
%|\{ x + \delta \nu_E(x), x\in A_\rho^+, \, \delta \in (0,\rho+\e\f(x)) \}  + r(\e) = \\
=\ \frac{1}{\e}\int_{\partial E} dx \int_0^{\e\f(x)} |\text{det}\big(I + (\rho +t) \nabla \nu_E(x)\big)| d t + o(1). 
%\frac{o(\e)}{\e} \\
%=\ \int_{\partial E}\kappa^+_\rho(E,x) \f \, dx  + o(1) %% \frac{o(\e)}{\e}.
% \end{split}
%\end{equation}
\end{multline*}
For $\e\to 0$ we recover \eqref{first variationp}.
\smallskip

\noindent\textbf{Step 2.}
In this step we show that the curvature $\kappa^+_\rho$ vanishes on $B_\rho^+$. This amounts to show that, if $\f$ has support in $B_\rho^+$, then 
$$
\frac{d}{d \e} \E^+_\rho\big(\Phi_\e(E)\big)_{|_{\e=0}} = 0.
$$
This is readily seen, since by definition of $B_\rho^+$ we have that, for $\e$ small enough,  $E^+_{\f,\e} = E^+_{\f, 0}$, so that 
$$
\E^+_\rho\big(\Phi_\e(E){| E}\big) =|E^+_{\f,\e}| = |E^+_{\f, 0}| = \E^+_\rho\big(E{| E}\big). 
$$
\smallskip

\noindent\textbf{Step 3.}  
In this step we conclude the proof by standard localization arguments.  Given  $\delta>0$, we can always write $\f= \f_1+\f_2+\f_3$ where supp$(\f_1)\subset A_\rho^+$, supp$(\f_2)\subset B_\rho^+$, $\hs^{d-1}(\text{supp}(\f_3))\le \delta$ and $|\f_i| \le |\f|$.  Notice that
$$
\E^+_\rho\big(\Phi_\e(E){| E}\big) =|E^+_{\f_1+\f_2+\f_3,\e}| = |E^+_{\f_1+\f_2,\e}| + r, 
$$
where $|r| \le C \e \delta $.  Moreover, { for $\e$ small enough} 
$$
E^+_{\f_1+\f_2,\e} =  E^+_{\f_1,\e}.
$$
Therefore, using Step 1, and since  $\kappa^+_\rho \f_2 \equiv 0$ by Step 2,  we conclude 

\begin{equation}
\begin{split}
\limsup_{\e\to 0} \left|
\frac{\E^+_\rho\big(\Phi_\e(E){| E}\big) - \E^+_\rho\big(E{| E}\big)}{\e}-
\int_{\partial E} \kappa^+_\rho \f \, d \hs^{d-1}  \right|\le
\\
\lim_{\e\to 0} \left| \frac{|E^+_{\f_1,\e}| - |E^+_{\f_1,0}| }{\e}  - \int_{\partial E} \kappa^+_\rho \f_1 \, d \hs^{d-1}\right| +  C'\delta =C'\delta. 
\end{split}
\end{equation}
We conclude by the arbitrariness of $\delta$.

%

%\begin{equation}
%\begin{split}
%\frac{d}{d \e} \E^+_\rho\big(\Phi_\e(E){| E}\big)_{|_{\e=0}} = \lim_{\e\to 0} \frac{\E^+_\rho\big(\Phi_\e(E){| E}\big) - \E^+_\rho\big(E{| E}\big)}{\e}=\\
%\lim_{\e\to 0}\frac{|E^+_{\f_1,\e}| - |E^+_{\f_1,0}| + r}{\e}  = \int_{\partial E} \kappa^+_\rho \f_1 \, d \hs^{d-1} +  r(\delta) =\\
% \int_{\partial E} \kappa^+_\rho (\f_1 + \f_2) \, d \hs^{d-1} +  r(\delta) =\int_{\partial E} \kappa^+_\rho \f \, d \hs^{d-1} +   r(\delta)+  r'(\delta), 
%\end{split}
%\end{equation}
%where $r(\delta), \, r'(\delta) \to 0$ as $\delta \to 0$. 
%
\end{proof}

Now we introduce the non-local curvature $\kappa_f$  associated to the energy \eqref{econeffe}:

\begin{equation}\label{kappaf}
\kappa_f(E,x) = \kappa^+_f(E,x) + \kappa^-_f(E,x), 
\end{equation}
where
\begin{equation}\label{kappapmf}
\kappa^\pm_f(E,x)   = \int_0^\rho (- 2s f'(s)) \kappa^\pm_s(E,x) \,ds.
\end{equation}

%When we want to highlight the dependence of the curvature $\kappa_\rho$ and $\kappa_f$ on the set $E$ we will write $\kappa^E_\rho$ and $\kappa^E_f$, respectively. 

The following result is a direct consequence of~\eqref{econeffe}, Lemma~\ref{curvatura} and Remark~\ref{nonlocvel}:
\begin{theorem} \label{thcurvatura}
Let $E\subset \R^n$ be an open bounded set of class $C^2$.
Then, for every $\f\in C^2 (\partial E;\R)$  we have 
\begin{equation}\label{first variationf}
\frac{d}{d \e} \E \big(\Phi_\e(E)\big)_{|_{\e=0}} =  \int_{\partial E} \kappa_f(E,x) \f(x) \, d \hs^{d-1}(x),
\end{equation}
where $\Phi_\e$ is a diffeomorphism such that  $\Phi(x)= x + \e\f(x) \nu_E(x)$ for $x\in\partial E$. 
\end{theorem}

\section{Viscosity solutions of the non-local level-set equation}
\label{viscoustheory}
In this section we introduce the level set formulation of the 
geometric evolution problem 
\begin{equation}\label{oee}
V = \kappa_f,
\end{equation}
where $V$ represents the normal velocity of the boundary of the evolving sets $t \mapsto  E_t$, and 
we give a proper notion of viscosity solution.  Then,  we develop an abstract setting where we provide existence and uniqueness results, for a  suitable class of Hamiltonians. Unfortunately, as already said in the introduction, this theory applies only to a suitable regularization of the curvature $\kappa_f$. However, we will also provide later on (Section~\ref{secgeometric}) an existence result,
without uniqueness, for equation~\eqref{oee}.
\subsection{The non-local evolution}\label{secnon-localevol}
Here we introduce the level set formulation of the  geometric evolution problem \eqref{oee}.
To this aim, following the level set approach, we identify $ E_t$ with the { superlevel set $\{u\ge 0\}$}  of a function $u:\R\times \R^d \mapsto \R$, and study the corresponding degenerate parabolic equation in $u$. 
%To this aim, 
Let $\Symd$ denote the class of $d\times d$ symmetric matrices, and $\Clo$ the class of closed subsets of $\R^d$.   We introduce the Hamiltonian $F_f:\R^d\times\R^d \times \Symd \times \Clo \mapsto \R$ defined by
\begin{equation}\label{Ff}
F_f(x,p,X,K) := \int_0^{\rho_0} (-2s f'(s))F_s(x,p,X,K)\,ds\,,
\end{equation}
where
\begin{equation}\label{Frho}
F_s(x,p,X,K)\ =\ F_s^+(x,p,X,K)\,+\,F_s^-(x,p,X,K)
\end{equation}
and the functions $F_s^+$ and $F_s^-$ are defined as follows:
\begin{equation*}
F_s^+(x,p,X,K)\ =\begin{cases}
\ds  \frac{|p|}{2 s}
\det \left [I - \frac{s}{|p|}\pro X\pro \right]^+
  &
\textrm{ if } \left\{
\begin{array}{ll} p\neq 0,\\ %% \, I-\frac{\rho}{|p|}\pro X\pro \ge 0,\\ 
\dist(x - s\hat{p},K)\ge s,
\end{array} \right. \\[2mm]
0 & \textrm{ otherwise,}
\end{cases}
\end{equation*}
\begin{equation*}
F_s^-(x,p,X,K)\ =\begin{cases}
\ds  -\frac{|p|}{2s}
\det \left[I+\frac{s}{|p|}\pro X\pro \right]^+
  &
\textrm{ if } \left\{
\begin{array}{ll} p\neq 0, \\%% \, I+\frac{s}{|p|}\pro X\pro \ge 0,\\ 
\dist(x+s\hat{p},K^c)\ge s,
\end{array} \right. \\[2mm]
0 & \textrm{ otherwise.}
\end{cases}
\end{equation*}
Here $\pro:=(I-\hat{p}\otimes\hat{p})$, where, for $p\neq 0$, $\hat p=p/|p|$,
and for $X$ a symmetric matrix, $[X]^+$ is the matrix with all
eigenvalues replaced with their positive part (in particular,
$\det[X]^+=0$ for any $X$ which is not positive definite).

\begin{remark}\textup{
If  $u$ is a smooth function and $u(x)$ is not a critical level,
by Theorem \ref{thcurvatura}
we easily deduce that
\begin{equation}\label{HamilCurv}
F_f(x,D u(x),D^2 u,\{y\in\R^d: u(y)\ge u(x)\}) = |D u(x)| \, \kappa_f({\{u\ge u(x)\}},x).
\end{equation}
In this identity we use in particular that, if $p=\nabla u$, $X= D^2 u$,  $K= \{u\ge u(x)\}$, then $\det \left[I - \frac{s}{|p|}\pro X\pro \right]^+ =0$  
means that there is a direction along which the curvature is {larger} than $1/s$,  so that $\kappa_s^+(K,x) = 0$.  
}
\end{remark}
% Next, we list some properties of the Hamiltonian, needed to prove the existence and uniqueness of the viscosity solution.

% \begin{theorem}
% The Hamiltonian $F_f$ defined in \eqref{Ff} satisfies the following properties
% \begin{itemize}
% \item[i)]Degenerate ellipticity: $F_f(x,p,X,K)\ge F_f(x,p,Y,K)$ if $X\le Y$;  
% \item[ii)]Monotonicity in the set variable: $F_f(x,p,X,K) \ge F_f(x,p,X,L)$ if $K\subseteq L$;   
% \item[iii)]$F_f$ is geometric: $F_f(x,t,\lambda p,\lambda X + \mu p \otimes p,K) = \lambda  F_f(x,t, p, X ,K)$ for all $\lambda\ge 0$, $\mu\in\R$.
% \item[iv)]Continuity: $F_f$ is continuous with respect to the product metric of $\R^d \times \R^d \times \Symd \times \Clo$, where $\Clo$ is endowed with the Hausdorff metric. 
% \item[v)]$F_f$ is uniformly Lipschitz: There exists $L>0$ such that $|F_f(x,t, p, X ,K) - F_f(y,t, p, X ,K)| \le |p| |X| $
% \end{itemize}
 
% \end{theorem}
% \begin{proof}
% In view of \eqref{Ffinte}, recalling $f'\le0$, the  degenerate ellipticity  follows if we prove that $F^\pm_{s}(x,p,X,K)\ge F^\pm_{s}(x,p,Y,K)$ for every $s>0$, whenever $X\le Y$. To show this inequality, notice that  $X\le Y$ yields $I+\frac{s}{|p|}\pro X\pro \le I+\frac{s}{|p|}\pro Y\pro$; the result easily follows by the increasing monotonicity of the determinant on positive semidefinite symmetric  matrices.
% \par
% Properties $ii)$ and $iii)$ follow by the very definition of $F_f$. 
% \end{proof}

The level set approach consists in solving the following parabolic
Cauchy problem 
\begin{equation}
\label{levelsetf}
\begin{cases}
\ds u_t(x,t) + F_f(x,D u(x,t), D^2 u(x,t), \{ y: \, u(y,t)\ge u(x,t)\}) \,=\,0 &\\
& \hspace{-2cm} \text{ for }t> 0\,, x\in \R^d, \\
 u(0,\cdot) = u_0,
\end{cases}
\end{equation}
in the viscosity sense. { The definition of a viscosity solution
 for such a non-local Hamiltonian  will be introduced in the next subsection. We will prove
an existence and uniqueness result in this setting, which
will be applied to a smoothed variant of $F_f$.}

%On the other hand, we will be able (in Section~\ref{secgeometric})
%to build evolutions which actually solve~\eqref{levelsetf}, by a well
%known time-discrete approach, but we have yet no uniqueness result.

\subsection{The abstract setting}

{  We introduce here a notion of viscosity solutions for problems
such as~\eqref{levelsetf}. }
%which is slightly modified with respect to the standard setting. 
The issues are of course that the Hamiltonian
is non-local, but also that it is singular in $p=0$ (at least in dimension
$d\ge 3$), in the sense that it
grows as the set vanishes as a power ($d-1$, in dimension $d$) of the curvature
tensor. For this reason, we have to adapt  {both} the setting 
of Slep\v{c}ev \cite{S} for non-local evolutions (notice however
that we will consider weaker continuity assumptions with respect to
the set variable), and the
one of Ishii and Souganidis \cite{IS} (see also  Goto~\cite{Goto}) for
singular Hamiltonians.

We will first list the properties which our Hamiltonians need to { satisfy}
in order to show an existence and uniqueness result, and then introduce
the appropriate definition of a viscosity solution (which is almost
standard).
Let $\Ope$ denote the family of open sets in $\R^d$, and $\Clo$ the family
of closed sets.
We consider Hamiltonians  $F:\R^{d}\times \R^{d}\setminus \{0\} \times \Symd  \times \{\Clo\cup\Ope\} \mapsto \R$   satisfying the following properties:
\begin{itemize}
\item[i)]Translational invariance: $F(x+r, p, X ,E + r) = F(x, p, X ,E)$ for every $r\in \R^d$;
\item[ii)]Degenerate ellipticity: $F(x,p,X,E)\ge F(x,p,Y,E)$ if $X\le Y$;  
\item[iii)]Monotonicity in the set variable: $F(x,p,X,E) \ge F(x,p,X,G)$ if $E\subseteq G$;
\item[iv)]Geometric property: $F(x, \lambda p,\lambda X + \mu p \otimes p,E) = \lambda  F(x, p, X ,E)$ for all $\lambda\ge 0$, $\mu\in\R$.
\item[v)]Continuity: $F$ is continuous with respect to its first variable,
moreover, the following properties hold:
\begin{itemize}
\item[v.1)] If $x_n\to x$, $p_n\to p\neq 0$, $X_n\to X$ and  $\{K_n\}\subset \Clo$ is a sequence converging to $K$ in the Kuratowski sense, then
$$
F(x,p,X,K)\le \liminf_n F(x_n,p_n,X_n,K_n).
$$ 
\item[v.2)]
If $x_n\to x$, $p_n\to p\neq 0$, $X_n\to X$ and  $\{A_n\}\subset \Ope$ is a sequence such that $A_n^c$ converges to $A^c$ in the Kuratowski sense, then
$$
F(x,p,X,A)\ge \limsup_n F(x_n,p_n,X_n,A_n).
$$ 
\end{itemize}
\item[vi)] There exists a continuous function $c: (0,+\infty)\mapsto (0,+\infty)$ such that, for all $x\in \R^d,\, p\in \R^d\setminus \{ 0\}, \, E\in \Clo\cup\Ope$ we have
\begin{equation}\label{contbound} 
-c(|p|) \le F(x,p,\pm I,E)  \le c(|p|).
\end{equation}
\end{itemize}
Following \cite{IS}, we introduce the family  $\F$ of functions $f\in C^2([0,\infty))$ such that $f(0) = f'(0)= f''(0)=0$, and such that $f''(r) >0$ for all $r>0$ which satisfy
\begin{equation}\label{IS}
\lim_{p\to 0} \frac{f'(|p|)}{|p|} c(|p|) = 0.
\end{equation}
We refer to  \cite[p. 229]{IS} for the proof  that the family $\F$ is not empty.

Let $T>0$ be fixed. As a slight variant to \cite{IS}, we introduce the
following definition.
\begin{definition}\label{defadmissible}
We will say that $\f\in C^{0}( \R^d\times (0,T))$ is
{\em admissible at the point $\hat z=(\hat x,\hat t)$}
if it is of class $C^2$ in a neighborhood of $\hat z$ and,
in case $D\f(\hat z)=0$, the following holds: 
there exists   $f\in\F$  and $\omega\in C^0([0,\infty))$ satisfying 
 $
 \lim_{r\to 0} {\omega(r)}/{r} = 0,
 $
such that
$$
|\f(x,t) - \f(\hat z) - \f_t(\hat z)(t-\hat t)|\le f(|x-\hat x|) + \omega(|t -\hat t|)
$$
for all $(x,t)$ in a neighborhood of $\hat z$.
\end{definition}

Given a function $u_0$, which is 
uniformly continuous in $\R^d$, we want to solve
\begin{multline}
\label{levelset}
 u_t(x,t) + F(x,D u(x,t), D^2 u(x,t),  \{  y\,:\,  u(y,t) \ge u(x,t)\})=0 
\\ \text{ for } (x,t)\in  \R^d  \times  (0,T),
\end{multline}
subject to the initial condition $u(0,\cdot) = u_0$.
We introduce the following definition of a viscosity sub/supersolution, 
inspired  from { both} frameworks of  \cite{IS} and \cite{S}.
\begin{definition}\label{defvisco}
An upper semicontinuous function $u:\R^d\times [0,T)\to \R$ is
a viscosity subsolution of~\eqref{levelset} if for all 
{ $z:= (x,t) \in \R^d\times (0,T)$ and
all $\f\in C^0(\R^d\times (0,T))$ such that $u-\f$ has a
maximum at $z$ and $\f$ is admissible at $z$ we have 
\begin{equation}\label{eqsubsol}
\begin{cases}
\displaystyle
\f_t(z)+ F\left(x,D\f(z),D^2\f(z),\{y: \f(y,t)\ge \f(z)\}\right)
\le 0 & \text{if } D\f(z)\neq 0,\\[2mm]
\f_t(z) \le 0 & \text{otherwise.}
\end{cases}
\end{equation}
A lower semicontinuous function 
is a viscosity supersolution of~\eqref{levelset} if for all $z\in\R^d\times (0,T)$ and
all  $\f\in C^0(\R^d\times (0,T))$ such that $u-\f$ has a
minimum at $z$ and $\f$ is admissible at $z$ we have
}
\begin{equation}\label{eqsupsol}
\begin{cases}
\displaystyle
\f_t(z)+F\left(x,D\f(z),D^2\f(z),\{y: \f(y,t) > \f(z)\}\right)
\ge 0 & \text{if } D\f(z)\neq 0,\\[2mm]
\f_t(z) \ge 0 & \text{otherwise.}
\end{cases}
\end{equation}

Finally, a function $u$ is a viscosity solution of~\eqref{levelset} if its upper semicontinuous envelope is a subsolution and 
 its lower semicontinuous envelope is a supersolution of~\eqref{levelset}.
 \end{definition}
As it is standard in the theory of viscosity solutions, { the  maximum in the definition of  subsolutions can be assumed to be strict, while the test functions $\f$ can be assumed to be coercive 
(and similarly for supersolutions).  
}
For the reader's convenience, we show that this is the case also in our non-local setting. Assume for instance that $u$ is a subsolution, 
{
$u-\f$ has a  maximum
}
 at some $(x,t)$, with $\f$ admissible at $(x,t)$. If $D \f(x,t)\neq 0$ we replace $\f$ with 
$$
\f_\e(y,s):= \f(y,s) + \e|y-x|^2 + |t-s|^2.
$$    
Then the maximum of $u-\f_\e$ at $(x,t)$ is strict, and we recover the inequality \eqref{eqsubsol} for $\f$ by letting $\e\to 0$ and using the semicontinuity of $F$, observing that the sets
$\{\f(y,t) + \e|y-x|^2\ge \f(x,t)\}$ converge 
{
to
 $\{\f(y,t) \ge \f(x,t)\}$ in the Kuratowski sense. We use
then the semicontinuity property v.1) to conclude.
}
\par 
If $D \f(x,t) = 0$,  we choose $f\in \F$ as in Definition~\ref{defadmissible},
and replace $\f$ by 
$$
\tilde \f(y,s):= \f(y,s) + f(y-x) +  |t-s|^2.
$$ 
We still have $D \tilde\f(x,t) = 0$,  $\tilde \f$ is  admissible at $(x,t)$, 
$\tilde\f_t(x,t)=\f_t(x,t)$ and  now the 
{ maximum 
}
of $u-\tilde \f$ is strict.

Notice that our definition of supersolutions and subsolutions is formally different  than the one given in \cite{S}, that involves  the superlevel
 sets of $u$ instead of $\f$. Indeed, in the case of a subsolution
we can assume that the test function $\f$ is such that $u\le \f$,
and $u(x,t)=\f(x,t)$.
Then, 
$\{y: u(y,t)\ge u(x,t)\}\subset \{y: \f(y,t)\ge \f(x,t)\}$ so that
\begin{multline*}
F\left(x,D\f(x,t),D^2\f(x,t),\{y: u(y,t)\ge u(x,t)\}\right)
\\ \ge\ 
F\left(x,D\f(x,t),D^2\f(x,t),\{y: \f(y,t)\ge \f(x,t)\}\right).
\end{multline*}
Therefore,  our definition seems actually weaker.
The following Lemma shows that, in fact, it is equivalent.
\begin{lemma}\label{laura}
Let $u$ be a viscosity subsolution of~\eqref{levelset}.
Then, for all $(x,t)$ in $\R^d\times (0,T)$ and
all  $\f\in C^0(\R^d\times (0,T))$ admissible at $(x,t)$, with $D\f(x,t)\neq 0$, and such that $u-\f$ has a
maximum at $(x,t)$ we have 
\begin{equation}\label{eqsubsolslepcev}
\f_t(x,t)\,+\,
F\left(x,D\f(x,t),D^2\f(x,t),\{y: u(y,t)\ge u(x,t)\}\right)
\ \le\ 0.
\end{equation}
A similar statement holds for supersolutions.
\end{lemma}
\begin{proof}
We can assume that the test function $\f$ is such that $u\le \f$
and $u(x,t)=\f(x,t)$.
Consider a decreasing sequence $\psi^n$ of functions which are smooth and
such that $\inf_n \psi^n=u$, $\psi^n\ge u+1/n$. Such a sequence exists
because $u$ is upper-semicontinuous. We consider now the test function
$\f^n=\min\{\f,\psi^n\}$, and notice that $\f^n = \f$ in a neighborhood of $(x,t)$, and hence $u-\f^n$  still has a maximum  at $(x,t)$.
By the very definition of subsolutions we have
\begin{equation*}%\label{eqsubsol}
\f_t(x,t)\,+\,
F\left(x,D\f(x,t),D^2\f(x,t),\{y: \f^n(y,t)\ge u(x,t)\}\right)
\ \le\ 0.
\end{equation*}
Consider
$K_n=\{\f^n(\cdot,t)\ge u(x,t)\}\supseteq K=\{u(\cdot,t)\ge u(x,t)\}$.
Since { the sequence of} the sets $K_n$ is nonincreasing, $K_n\to \bigcap_{k} K_k$ in
the Kuratowski sense, and by construction $K=\bigcap_{k} K_k$. It follows
that
\begin{multline*}
\liminf_{n\to \infty}
 F\left(x,D\f(x,t),D^2\f(x,t),\{y: \f^n(y,t)\ge u(x,t)\}\right)
\\ \ge\  F\left(x,D\f(x,t),D^2\f(x,t),\{y: u(y,t)\ge u(x,t)\}\right)
\end{multline*}
and~\eqref{eqsubsolslepcev} follows.
\end{proof}

\begin{remark}
\label{compmon}
{\rm
By the assumption iv) on $F$,  a standard argument shows that if $u$ is a subsolution (supersolution) and $\theta:\R\to R$ is increasing, then $\theta \circ u$ is still a subsolution (supersolution).  
}
\end{remark}

\subsection{A comparison result}
Here we provide a comparison result, that is the main ingredient  to get existence and uniqueness in the viscosity setting.
Let us set
$$
Q_T:= \R^d \times (0,T), \qquad \partial_p Q_T = \R^d\times \{0\}, \qquad  \RT=Q_T\cup \partial_p Q_T.
$$
Moreover, we denote by $USC(\RT)$ and $LSC(\RT)$ the space of upper and lower semicontinuous functions on $\RT$, respectively.  
The following comparison principle is an extension of \cite[Theorem 1.7]{IS} for non local evolutions.

\begin{theorem} 
Let $u\in USC(\RT)$ and $v\in LSC(\RT)$ be a subsolution and a supersolution of \eqref{levelset}, respectively. Assume that
\begin{equation}\label{disbordo}
\lim_{r\downarrow 0} \sup \{ u(z) - v(\zeta): \, (z,\zeta) \in \partial_pQ_T\times \RT) \cup (\RT\times \partial_p Q_T), \, |z-\zeta|\le r \} \le 0.
\end{equation} 
Then $u\le v$ in $\RT$, and moreover,
$$
\lim_{r\downarrow 0} \sup \{ u(z) - v(\zeta): \, z,\zeta \in \RT, \, |z-\zeta|\le r \} \le 0.
$$
\end{theorem}
\begin{proof}
The proof follows the line of the proof of \cite[Theorem 1.7]{IS}.
{
We do not provide a self-contained proof; 
we only indicate the changes needed to adapt that proof 
to the context of our  non-local setting. 
}
For the reader's convenience,
we will use the same notation as in~\cite{IS}, up to the fact that
in our case, $\Omega=\R^d$ (and the space dimension is denoted by $d$ instead
of $N$).
\par
As in \cite{IS}, by Remark \ref{compmon} we may assume without loss of generality that $u$ and $v$ are bounded, and we extend their domain of definition on $\overline Q_T$ by setting
\begin{align*}
& u(x,T) \ =\  \lim_{r\downarrow 0}\, \sup\{ u(y,s)\,:\,(y,s)\in \RT\,, 
|y-x|+|s-T|\le r\}\,, \\
& v(x,T) \ =\  \lim_{r\downarrow 0}\, \inf\{ v(y,s)\,:\,(y,s)\in \RT\,, 
|y-x|+|s-T|\le r\}\,.
\end{align*}
The functions $u$ and $v$ are still upper and lower semicontinuous in $\overline Q_T$, respectively. 

We first show that  $u$ is still a subsolution (and $v$ a supersolution) in $\R^d\times(0,T]$ in the obvious sense. Assume indeed that $u-\f$ has a strict maximum at $z=(y,T)$, where $\f$ is admissible 
{ and  coercive.
}
\par
Assume first that  $D\f(z)\neq 0$. 
For any $n\in \N$ large enough, the function $(x,t)\mapsto u(x,t) - \f(x,t) - 1/[n(T-t)]$ attains a  maximum at a point $z_n = (y_n,t_n) \in Q_T$, where $z_n\to z$ as $n\to \infty$, moreover we have $D\f(z_n)\neq 0$ for $n$ large. Hence, 
$$
\f_t(z_n) +  \frac{1}{n(T-t_n)^2}+ F(x_n, D\f(z_n), D^2\f(z_n), \{\f(\cdot,t_n)\ge \f(z_n)\}) \le 0. 
$$
Since any Kuratowski limit of  $\{\f(\cdot,t_n)\ge \f(z_n)\}$ is contained in  $\{\f(\cdot,t)\ge \f(z)\}$,  using properties iii) and v.1) of $F$ we deduce  
$$
\f_t(z) + F(x, D\f(z), D^2\f(z), \{\f(\cdot,t)\ge \f(z)\}) \le 0. 
$$

If now  $D\f(z) = 0$, we follow the lines of \cite[Proposition 1.3]{IS}. 
Since $\f$ is admissible at $z=(y,T)$, there are $\delta>0, \, f\in\F$ and $\omega \in C^0(\R)$ with $\omega(r)/r\to 0$ as $r\to 0$ such that
$$
|\f(x,t) - \f(z) - \f_t(z)(t-T)|\le f(|x-y|) + \omega(t-T)
$$ 
for all $(x,t)\in B(z,\delta)$. Without loss of generality we assume that $\omega \in C^1(\R)$ and $\omega(0) = \omega'(0)=0$ and also that $\omega(r)>0$ for $r\neq 0$.
Next choose a sequence $\omega_n\in C^2(\R)$ such that $\omega_n(r)\to \omega(r)$ and $\omega'_n(r)\to \omega'_n(r)$ locally uniformly in $\R$ and set 
\begin{align}
& \psi(x,t)=\f_t(z)(t-T)+2f(|x-y|)+2\omega(t-T)\,,\nonumber\\
&  \psi_n(x,t)=\f_t(z)(t-T)+2f(|x-y|)+2\omega_n(t-T)-\frac{1}{n(T-t)}\,. \label{psienne}
\end{align}
%{
%TUTTI I 2 MI SEMBRA NON SERVANO}
We have $ u-\psi$ has a 
{ strict maximum at $z$.
}
 Hence for $n$ large enough $ u-\psi_n$ has a 
 {
  strict maximum
  }
   at $z_n=(y_n,t_n) \in Q_T$, with $z_n\to z$, and 
$\psi_n$ is admissible at $z_n$. As $u$ is  a subsolution, we have, using also property iv) of $F$,
\begin{multline}\label{boh}
\f_t(z)+2\omega'_n(t_n-T)+\frac{1}{n(T-t_n)^2}
\\
+  \frac{2f'(|y_n-y|)}{|y_n-y|}F(y_n, {y_n-y},  I, \{\psi_n(\cdot, t_n)\geq \psi_n(z_n)\})\leq 0
\end{multline}
if $y_n\neq y$, while $\f_t(z)+2\omega'_n(T-t_n)+1/[n(T-t_n)^2]\leq 0$ if $y_n=y$. Letting $n\to \infty$, we get $\f_t(z)\leq 0$ thanks to \eqref{IS}.
Hence, as claimed, $u$ is a subsolution in $\R^d\times (0,T]$.
\smallskip

Now, as in \cite{IS}, we assume that 
\begin{equation}
\theta_0\,:=\,
\limsup_{r\downarrow 0} \left\{ u(z)-v(\zeta)\,:\, (z,\zeta)\in \overline{Q}_T^2,
|z-\zeta|\le r \right\} \, >\, 0
\end{equation}
and try to get a contradiction. The proof then follows identically the proof
in~\cite{IS} from page 238 until the middle of page 241. In particular
(using exactly the same notation), the case ``$\hat{\theta}=\theta$''
is identical (since the non-locality does not play any role
in that case), and we may jump to the case $\hat{\theta}<\theta$.
As in \cite{IS},
we then let
$(\hat x, \hat t, \hat y, \hat s) \in \overline Q_T \times \overline Q_T$
be the maximum point of 
\begin{equation}\label{defhat}
u(x,t) - v(y,s) -\alpha f(|x-y|) - \alpha (t-s)^2 - \e t  - \e s - \delta |x|^2 - \delta |y|^2,
\end{equation}
where $f\in\F$, and $\e,\, \alpha>0$  are suitable positive constants,
and $\delta>0$ is chosen in such a way that the value of this
maximum point is strictly positive. We then can jump to
the middle of page~241
(more precisely, up to ``Now the definition of viscosity solution yields'').
Here, the situation is a bit changed. By Lemma \ref{laura} we get
$$
2\alpha (\hat t- \hat s) + \e + F\left(\hat x, \alpha f'( |\hat p|) \frac{\hat p }{|\hat p|} + 2\delta \hat x, X+ 2\delta I, \{ u(\cdot,\hat t) \ge u(\hat x, \hat t)\}\right) \le 0,
$$
and
$$
2\alpha (\hat t- \hat s) - \e + F\left(\hat y, \alpha f'( |\hat p|) \frac{\hat p }{|\hat p|} - 2\delta \hat y, X- 2\delta I, \{ v(\cdot,\hat s) > v(\hat y, \hat s)\}\right) \ge 0 .
$$
Here $X$ is a suitable symmetric matrix which also depends on $\delta$, and $\hat p:=\hat x - \hat y$. Moreover, $X$, $\hat p$, $\hat t$ and $\hat s$ 
are uniformly bounded, while $|\hat p|$ is bounded from below. 
Therefore, we may assume that they converge, as $\delta \to 0$, to some limit denoted in \cite{IS} by $Y, \, \bar p\neq 0, \, \bar t, \, \bar s$, respectively.    
Denote
$$
K_\delta:= \{ u(\cdot,\hat t) \ge u(\hat x, \hat t) \}- \hat x, \qquad L_\delta:=  \{ v(\cdot,\hat s) > v(\hat y, \hat s) \}- \hat y.
$$
We may also assume that $K_\delta \to K$, $L_\delta^c \to L^c$ in the Kuratowski sense, for some $K\in\Clo$, $L\in\Ope$. We deduce
(using the semicontinuity properties of $F$ and the translational
invariance) that
\begin{equation}\label{iltribunalenonemale}
\begin{array}{l}
\displaystyle 2\alpha (\bar t- \bar s) +  \e
+ F\left(0, \alpha f'( |\bar p|) \frac{\bar p }{|\bar p|} , Y, K \right) \le 0,
\\[3mm]
\displaystyle 2\alpha (\bar t- \bar s) - \e
+ F\left(0, \alpha f'( |\bar p|) \frac{\bar p }{|\bar p|} , Y, L\right) \ge 0.
\end{array}
\end{equation}
By \eqref{defhat} we have
\begin{multline}\label{maxmix}
 u( x,\hat t) - u(\hat x,\hat t) \le  v(y,\hat s) - v(\hat y ,\hat s) - \\
 \left(  \alpha f(|\hat x-\hat y|) - \alpha f(|x-y|) + \delta |\hat x|^2 - \delta |x|^2 + \delta |\hat y|^2 - \delta |y|^2\right).
\end{multline}
Let $R>0$ and choose $\xi\in K_\delta \cap B_R$.
Let also $\eta\in (0,1/2)$
and $q= 2\eta \hat p$ (recall $\hat p=\hat x-\hat y$).
Choose $z$ with $|z|\le \eta|\hat p|$. Choosing $x=\hat x+\xi$
and $y=\hat y+\xi+q+z$ in~\eqref{maxmix}, and observing that
$x-y = (1-2\eta)\hat p -z$ so that $|x-y|\le (1-\eta)|\hat p|$, we
obtain since $\xi\in K_\delta$
\begin{eqnarray*}
0\le u(\hat x + \xi, \hat t) - u(\hat x,\hat t)
 \le v(\hat y + \xi + q + z, \hat s) - v(\hat y, \hat s) -\\
\Bigl(\alpha f(|\hat p| ) - \alpha  f((1-\eta) |\hat p|) - \delta \xi \cdot (2\hat x + \xi) - \delta(\xi + q+z) \cdot (2\hat y + \xi + q+z)
\Big).
\end{eqnarray*}
Since $\delta(|\hat x| + |\hat y|) \to 0$ (see \cite{IS}), $|\xi|\le R$, and
$$
\alpha f(|\hat p| ) - \alpha  f((1-\eta) |\hat p|)
\ \ge c >\ 0
$$
for some $c$ independent of $\delta$ (as $|\hat p|$ is bounded away from zero), we have 

$$
\alpha f(|\hat p| ) - \alpha  f((1-\eta) |\hat p|) - \delta \xi \cdot (2\hat x + \xi) - \delta(\xi + q+z) \cdot (2\hat y + \xi + q+z)>0
$$
for $\delta$ small.
Thus, 
$\xi+q+z\in L_\delta$. As this is true for all $|z|\le \eta|\hat p|$,
we find that for $\delta$ small enough, 
\[
q+(K_\delta\cap B_R(0)) + B_{\eta|\hat p|}(0)\ \subseteq \ L_\delta.
\]
In other words, the sets $q+(K_\delta\cap B_R(0))$ are at
distance at least $\eta|\hat p|$ from $L_\delta^c$.
Taking the (Kuratowski) limits as $\delta\to 0$ we deduce that
$\dist(2\eta \bar p+K,L^c)\ge \eta|\bar p|$, 
{ and in particular  that 
$2\eta\bar p+ K\subset L$. 
}
%(and $K\subset \overline{L}$, but this is not sufficient to conclude). 
Using property iii) of $F$ and
the translational invariance again, we deduce that
\[
F\left(-2\eta \bar p, \alpha f'( |\bar p|) \frac{\bar p }{|\bar p|} , Y, K \right) \ge
F\left(0, \alpha f'( |\bar p|) \frac{\bar p }{|\bar p|} , Y, L \right)
\]
for any $\eta\in (0,1/2)$.
Taking the limit $\eta\to 0$, and using the continuity of $F$
with respect its first variable together with~\eqref{iltribunalenonemale},
we obtain that $2\e\le 0$, a contradiction.
\end{proof}

\subsection{Existence and uniqueness of viscosity solutions}
To show the existence of viscosity solutions, we need the following
stability result. 
\begin{proposition}\label{propstab}
Let $(u_n)_{n\ge 1}$ be a sequence of upper semicontinuous subsolutions
of~\eqref{levelset} 
and let, for any $z=(x,t)$,
\[
u^*(z)\ =\ \lim_{r\downarrow 0} \sup\left\{ 
u_n(\zeta)\,:\, |z-\zeta|\le r\,, n\ge \frac{1}{r}\right\}.
\]
Then $u^*$ is also a subsolution of~\eqref{levelset}.
\end{proposition}
Of course, a symmetric result holds for supersolutions.
\begin{proof}
The proof of this result is a variant of the proof of \cite[Prop. 1.3]{IS}
(see also the proof of property \textbf{(P2)} in \cite{S}),
observing that if $z_n=(x_n,t_n)\to z=(x,t)$ and $\f$ is a test
function, then the sets $K_n:=\{\f(\cdot,t_n)\ge \f(z_n)\}$ converge
(up to a subsequence) in the Kuratowski sense to a set  $K\subseteq \{f(\cdot,t)\ge \f(z)\}$. We
conclude using the monotonicity and the semicontinuity properties of $F$.
\end{proof}

Given $A  \subset  \R^n$, we denote by $BUC(A)$ the  space of bounded, uniformly continuous functions from $A$ to $\R$. 
We now can state a general existence and uniqueness result:
\begin{theorem}\label{thexun}
Let $u_0\in BUC(\R^d)$.
Then, there exists a unique viscosity 
solution $u\in BUC(\R^d\times [0,\infty))$ of~(\ref{levelset}) with initial condition $u_0$.
\end{theorem}
%%% UC in time
\begin{proof}
The proof of this result is very classical, see~\cite{CIL,IS} and
based on Perron's method. We introduce
%\begin{multline*}
\[
\bar u(x,t) =
 \sup\Big\{ u(x,t)\,:\, u \textrm{ subsolution of \eqref{levelset}},
\, \min u_0\le u\le \max u_0\,, u(\cdot,0)\le u_0\Big\},
\]
%\end{multline*}
and $u^*,u_*$, its upper and lower semicontinuous envelopes.
The fact that $u^*$
is a subsolution follows from Proposition~\ref{propstab},
observing that at each point $(x,t)$ we can find a suitable sequence
of subsolutions $(u_n)_{n\ge 1}$ whose relaxed upper limit is $u^*(x,t)$.

The fact that $u_*$ is also a supersolution is classical and obtained
by contradiction, assuming that at some point $\bar z=(\bar x,\bar t)$
of (strict) contact with a test function $\f\le u_*$, $\f$ does not
satisfy~\eqref{eqsupsol}. 
If $D\f(\bar z)\neq 0$, one can use the test function $\f$ to construct a new subsolution $\bar u> u_*$ in a neighborhood of $\bar z$, thus contradicting the maximality of $u_*$. 
To treat the case $D\f(\bar z)=0$ one repeats the same construction, but  (as  in the proof of \cite[Prop. 1.3]{IS})   with $\f$ replaced by
%%\footnote{ GIUSTO? Ma perch\'e no? L'avevamo preso in \cite{IS}...}
\[
\psi(x,t)=\f(\bar z)
+ \f_t(\bar z)(t-\bar t)-2f(|x-\bar x|)-2\omega(t-\bar t).
\]

The regularity properties  of $u$ and the fact that the initial condition is attained, 
can be shown as in the last part of the proof of \cite[Theorem 1.8]{IS}.
%The regularity properties are standard: $u(\cdot, t)$ has
%the same spatial
%modulus of continuity $\omega$ as $u_0$ because of the translational
%invariance of the equation and the comparison principle. The uniform
%continuity in time follows from comparison with appropriate barriers
%from above and below (and using the uniform continuity in space),
%of the form
%\begin{equation}\label{barrier}
%a(\alpha) + \alpha f(|x-\bar x|) + b(\alpha)(t-\bar t)
%\end{equation}
%for $t\ge \bar t$. Here, $\alpha$ is a large constant,
%$a(\alpha)$ the smallest positive constant such that 
%\[
%\omega(|x-\bar x|)\le a(\alpha) + \alpha f(|x-\bar x|)
%\]
%for all $x$, and $b(\alpha)$ a constant large enough, so that
%\eqref{barrier} defines a supersolution for $t>\bar t$. By comparison,
%it will stay above $u$, and we deduce that
%\[
%\tilde{\omega}(s)\ =\ \inf_{\alpha} a(\alpha) + s\,b(\alpha)
%\]
%is a modulus of continuity for $u$.
%Similar arguments shows the continuity at time $t=0$
%(see the proof of Thm.~1.8 in~\cite{IS}).
\end{proof}

\subsection{Application to our evolution problem} % to $\E$}
Here we show how to apply the  viscosity approach developed above to our specific problem. 
First, we extend the Hamiltonian $F_f$ in \eqref{Ff} to open sets, enforcing that the evolution of a closed set $K$ agrees with the evolution of its complement, i.e., setting  for every $A\in \Ope$:
{
\begin{equation}\label{Ffopen}
F_f(x,p,X,A) :=  - F_f(x,-p,-X,A^c)  =  F_f(x,p,X,A^c),
\end{equation}
where the last identity follows by the very definition \eqref{Ff}, \eqref{Frho} of $F_f$.
}
% then we observe that such Hamiltonian 
%does not satisfy the continuity  assumptions which are required
%in Theorem~\ref{thexun}, and then we suitably regularize it to deduce existence and uniqueness 
%for an approximated evolution.  
It turns out, however, that the Hamiltonian $F_f$ in \eqref{Ff}
does \textit{not} satisfy all the assumptions which are required
in Theorem~\ref{thexun}. In fact, 
it lacks assuptions v), v.1), v.2).
A basic counterexample is as follows: Let $K$ be a ball, 
and $x_n\to x\in \partial K$ with $x_n\not\in K$ for all $n$.
Then, $F^-_\rho(x_n,p,X,K)=0$ for all $n$ and any $p,X$, while
if $p$ is the inner normal to $\partial K$, { $X$ is small enough  and 
 the radius of the ball $K$ is large enough, then
}
$F^-_\rho(x,p,X,K)<0$. On the other hand, 
{
%if $K$ is a ball large enough, then $\dist(x_n-\rho\hat{p},K)\ge \rho$
%for $n$ large enough, so that 
$F^+_\rho(x_n,p,X,K)$ will be constant
(and positive). 
}
Hence $F_f$, in that case, will be l.s.c., but not
continuous, and in particular v.2) does not hold, neither the
continuity { with respect to $x$}.

In fact, we observe now  that a continuous Hamiltonian extending
the non-local curvature~\eqref{kappaf} does not exist. Indeed,
let $K= \overline B:=\overline {B_1(0)}\subset \R^2$ and $x\in\partial B$. Let moreover 
$A_n$ be open smooth subsets of  $\bar B$ and $x_n\in \partial A_n$ be satisfying the following properties:
1) $A_n$ have vanishing diameter; 2)   $x_n \to x$; 3) 
{ 
the outer normal and the (euclidean) curvature of $A_n$ at $x_n$   agree with the outer normal and curvature of $B$ at $x$, respectively. 
These conditions are clearly compatible. Set $K_n:= B\setminus A_n$.}
Then 
$\kappa^\pm_f(K_n,x_n)=0$
for $n$ large enough (remember that $f'=0$ near $0$). The idea now is that, 
if we could extend $\kappa_f$ into a semi-continuous Hamiltonian
in the sense of v), it would follow that { $\kappa_f(K,x) \le 0$},
which is not true.  More precisely, let $u_n = -d_{K_n}$. By 
\eqref{HamilCurv} and since $Du_n(x_n) = Du(x)$,  $D^2u_n(x_n) = D^2u(x)$ we have
{
\begin{multline}
F_f(x,D u(x),D^2 u(x),K)  = \kappa_f(K,x)> 0 = \liminf_{n\to\infty} \kappa_f(K_n,x_n)
\\  
= \liminf_{n\to\infty}  F_f(x_n,D u_n(x_n),D^2 u_n(x_n),K_n)  
= 
\liminf_{n\to\infty} 
F_f(x_n,D u(x),D^2 u(x),K_n).
\end{multline}
}
Since $x_n\to x$ and $K_n\to K$, we conclude that property v.1) does not hold.
This means that 
Theorem~\ref{thexun} does not apply for our particular problem, without
further smoothing (see Proposition~\ref{Hamilsmooth} below).

We can show a continuity slightly weaker than 
assumptions  v), v.1), v.2), which however will have some utility
in the sequel. The following result shows that these properties are
essentially
true if $(x,p,X,K)$ are of the form $(x,D\f(x),D^2\f(x),\{\f\ge \f(x)\})$,
when $D\f(x)\neq 0$. Finding a result similar to Theorem~\ref{thexun}
but under this weaker assumption would be very interesting, and
is a subject for future study.
%It might be, then, that $F_f$ is a ``bad'' extension of the curvature
%away from arguments of this form, however, we were not able yet to
%find a better one.
\begin{lemma}
Let $\f_n,\f:\R^d\to \R$ be $C^2_{loc}$ functions,
and assume that $\f_n\to \f$ in $C^2_{loc}$ as $n\to\infty$.
Let $x\in \R^d$ with
$D\f(x)\neq 0$ and consider a sequence $(x_n)$ with $x_n\to x$. Then,
\begin{multline}\label{Ffsci}
F_f(x,D\f(x),D^2\f(x),\{\f\ge {\f(x)}\}) \\ \le\ \liminf_{n\to\infty}
F_f(x_n,D\f_n(x_n),D^2\f_n(x_n),\{\f_n\ge \f_n(x_n)\})
\end{multline}
and
\begin{multline}\label{Ffscs}
F_f(x,D\f(x),D^2\f(x),\{\f> {\f(x)}\}) \\ \ge\  \limsup_{n\to\infty}
F_f(x_n,D\f_n(x_n),D^2\f_n(x_n),\{\f_n> \f_n(x_n)\})\,.
\end{multline}
\end{lemma}
\begin{proof}
We prove only the first inequality, the second one being a consequence of the first one and the identity 
$$
{ F_f(x,D\f(x),D^2\f(x),\{\f> \f(x)\}) = -F_f(x,-D\f(x),-D^2\f(x),\{-\f\ge \f(x)\}).}
$$
First of all, replacing $\f_n$ with $y\mapsto \f_n(y-x+x_n)$  we may assume (by translation invariance of the Hamiltonian) that $x_n=x$ for all $n\ge 1$. Denote
$p=D\f(x)$, $p_n=D\f_n(x)$,
$\hat{p}=D\f(x)/|D\f(x)|$, and $\hat{p}_n=D\f_n(x)/|D\f_n(x)|$,
$X=D^2\f(x)$, $X_n=D^2\f_n(x)$,  $K=\{\f\ge \f(x)\}$, $K_n=\{\f_n\ge \f_n(x)\}$.
One has that $p_n\to p$, etc, except for one detail: $K_n$ may not
converge to $K$: more precisely any Kuratowski limit of a subsequence of
$K_n$ is a set in between $\{ \f>\f(x)\}$ and $K$.

Consider now $s\in [\underline{\delta},\delta]$ 
{
(recall that $f$ is constant on $[0,\underline\delta]$).
} Since $K$ is $C^2$ near
$x$ (by the implicit function theorem), there exists a
 positive  $s^* \in (0,\delta]$, such that $\dist(x-s\hat{p},K)=s$ if $s\le s^*$,
and $\dist(x-s\hat{p},K)<s$ if $s\in (s^*,\delta]$ (possibly empty).

We want to prove that for a.e.~$s$ (in fact, for all $s\neq s^*$),
%In the latter case $\kappa^+_s(K,x)= 0$ and clearly
\begin{equation}\label{kappalinf}
\kappa^+_s(K,x)\ \le\ \liminf_{n\to\infty} \kappa^+_s(K_n,x).
\end{equation}
Since, clearly,
\[ 
\frac{|p_n|}{2s}
\det \left[I-\frac{s}{|p_n|} \mathcal{P}_{p_n}X_n\mathcal{P}_{p_n}\right]^+
\ \stackrel{n\to\infty}{\longrightarrow}
\ \frac{|p|}{2s}\det \left[I-\frac{s}{|p|} \pro X \pro\right]^+\,,
\]
we need to show \eqref{kappalinf} only when the right-hand side
is zero, or more precisely, when $\dist(x-s\hat{p}_n,K_n)<s$ for infinitely many
$n\ge 1$. 
In this case, let $n_k$
be a subsequence such that 
{ 
$K_{n_k} \to \tilde{K}$
}
(Kuratowski) and $\dist(x-s\hat{p}_{n_k},K_{n_k})<s$ for all $k$. Let 
$y_{n_k} \in K_{n_k}$ such that $|x-s\hat{p}_{n_k}-y_{n_k}|<s$, and we can also
assume that $y_{n_k}\to y\in \tilde{K}\subset K$. There are two situations:
\begin{itemize}

\item either $y\neq x$, in which case, since $|x-s\hat{p}-y|\le s$,
we have $s\ge s^*$. 
%(either $s>s^*$, or $s$ might be $s^*$
% in case the latter distance is exactly $s$ and $x -s\hat{p}$ has at least
%two projections $x$ and $y$ on $K$). 
Since for $s > s^*$ $\kappa^+_s(K,x)= 0$, we conclude that \eqref{kappalinf} holds for all  $s\neq s^*$;

\item or $y=x$, in which case there exists $z_k\in [x,y_{n_k}]$ such that
%\begin{multline*}
\[
\f_{n_k}(x) \le \f_{n_k}(y_{n_k}) = \f_{n_k}(x) + D\f_{n_k}(x)\cdot(y_{n_k}-x)
+ \frac{1}{2}(D^2\f_{n_k}(z_k)(y_{n_k}-x))\cdot(y_{n_k}-x),
\]
%\end{multline*}
hence
\begin{equation}\label{A1}
0\ \le\ p_{n_k}\cdot (y_{n_k}-x) \,+\, \frac{1}{2} (X_{n_k}(z_k) (y_{n_k}-x))\cdot(y_{n_k}-x).
\end{equation}
Now, 
\begin{equation*}
s^2\, >\, |x-s\hat{p}_{n_k}-y_{n_k}|^2\,=\,|y_{n_k}-x|^2+s^2 + \frac{2s}{|p_{n_k}|}
(p_{n_k}\cdot (y_{n_k}-x)),
\end{equation*}
and hence,   dividing by  $|y_{n_k}-x|^2$ we have
\begin{equation}\label{A2}
0 > 1+ \frac{2s}{|p_{n_k}|} \frac{(p_{n_k}\cdot (y_{n_k}-x))}{|y_{n_k}-x|^2},
\end{equation}
Set $\xi_{k} = (y_{n_k}-x)/|y_{n_k}-x|$. Up to a subsequence  $\xi_k\to \xi\perp \hat p$. By \eqref{A1} and \eqref{A2} we conclude
\[
\frac{s}{|p|} X\xi\cdot\xi \ \ge \ 1
\]
and in particular, $I-s'/|p|\pro X \pro$ has a negative eigenvalue as
soon as $s'>s$. It follows that $s^*\le s$ and, again,  we deduce
\eqref{kappalinf} for all  $s\neq s^*$.
\end{itemize}
Now, it remains to
take the integral for $s\in [\underline{\delta},\delta]$, and
it follows, using Fatou's lemma, that:
%\begin{multline*}%\label{Ffsci}
\begin{multline*}
F_f^+(x,D\f(x),D^2\f(x),\{\f\ge \f(x)\}) \\ \le\ \liminf_{n\to\infty}
F_f^+(x,D\f_n(x),D^2\f_n(x),\{\f_n\ge \f_n(x)\})\,.
\end{multline*}
%\end{multline*}
In order to show \eqref{Ffsci}, it remains to show a similar inequality
for $F_f^-$.
\medskip

This time, we need  to show that for almost any
$s\in [\underline{\delta},\delta]$, if, for a subsequence,
$\kappa_s^-(K_{n_k},x) <0 $, then
$\kappa_s^-(K,x)$ must
also take the value $-1/(2s)\det[I+(s/|p|)\pro X \pro]^+$.
But it means precisely that $\dist(x_{n_k}+s\hat{p}_{n_k},K_{n_k}^c)=s$
for all $k$, hence $\f\ge \f(x_{n_k})$ on the ball of
center $x_{n_k}+s\hat{p}_{n_k}$ and radius $s$. Passing to the limit,
we deduce that $\f\ge \f(x)$ on the ball of center $x+s\hat{p}$ and
radius $s$, so that $\dist(x+s\hat{p},K^c)\ge s$. The thesis follows.
\end{proof}

{
Finally, we build an approximation of the Hamiltonian
$F_f$
which will fulfill the assumptions (i--v) of Section~\ref{viscoustheory}.
To this purpose it is clearly enough to  approximate $F_\rho$ for fixed $\rho$; a possibility is a follows, 
}
\begin{multline}\label{appf}
F_\e(x,p,X,E)\ :=\ \frac{|p|}{2\rho}
\det \left[ I-\frac{\rho}{|p|}\pro X\pro \right]^+
H_\e( \dist(x-\rho \hat{p},E)-\rho)\\
-\ \frac{|p|}{2\rho}
\det \left[ I+\frac{\rho}{|p|}\pro X\pro \right]^+
H_\e( \dist(x+\rho \hat{p},E^c)-\rho),
\end{multline}
where $H_\e(t)$ is a continuous approximation of the Heavyside function,
which is $1$ for $t\ge 0$, $0$ for $t\le -\e$, and nondecreasing.
%Here, for a symmetric matrix $X$, $[X]^+$ denotes the matrix $X$
%with all eigenvalues  replaced with their positive part, in particular,
%its determinant is zero if $X$ is not positive definite.

In that case, 
\begin{equation}\label{cippi}
|F_\e(x,p,\pm I,E)|\le \frac{|p|}{2\rho}
\left(\left(1+\frac{\rho}{|p|}\right)^{d-1}+
\left[\left(1-\frac{\rho}{|p|}\right)^+\right]^{d-1}
\right)\le c(|p|) \sim |p|^{2-d}
\end{equation}
as $|p|\to 0$, and in dimension $d\ge 3$ this Hamiltonian is
indeed singular. The following result is straightforward:
\begin{proposition}\label{Hamilsmooth} 
{\rm
The Hamiltonian $F_\e$ satisties
all the properties required in Section~\ref{viscoustheory}.
}
\end{proposition}
%\begin{proof}
%is there any issue with the ellipticity???
%\end{proof}

In particular, by Theorem \ref{thexun} we deduce  existence and uniqueness, in the  viscosity setting,  of  the geometric flow
corresponding to the regularized  non local curvature \eqref{appf}. 
In the next section we show an existence result (but with no proof
of uniqueness) for the original non-local flow~\eqref{levelsetf}.

\section{The geometric evolution associated to $\E^f$}
\label{secgeometric}

We will now follow a different approach, in order to construct a
(level-set) flow
of our curvature which is actually a viscosity solution of the equation
\eqref{levelsetf}, with
% \begin{equation}\label{levelseteq} %% identical to levelsetf?
% \begin{cases}\ds
% \frac{\partial u}{\partial t} + F_f(x,u,Du,\{u \ge u(x)\})\ =\ 0
% \\ t>0, x\in \R^d
% u(0,\cdot)\ =\ u_0& \textup{  in }\R^d
% \end{cases}
% \end{equation}
$F_f$ the Hamiltonian defined in~\eqref{Ff}. Here
we assume that $u_0$ is an initial datum 
with compact support, and bounded, uniformly continuous ($u_0\in BUC_c(\R^d)$).

The construction follows the approach first suggested by Luckhaus and Sturzenhecker,
and Almgren, Taylor and Wang~\cite{LS95,ATW93}. We follow here a simple
strategy which has been elaborated in~\cite{CiomagaThouroude}
for the classical Mean Curvature Flow, and which we adapt to our setting.

First, given a time-step $h>0$  and a compact set $E$, we define $T^-_hE$ (resp, $T^+_h E$) as the
minimal (resp., maximal) solution to
\begin{multline}\label{NLATW}
\min_{F\subset \R^d} \left\{\E^f(F)\ +\ \frac{1}{h}\int_{F\triangle E} \dist(x,\partial E)\,dx\right\}
\ \\
=\ \min_{F\subset \R^d}\left\{ \E^f(F)\ +\ \frac{1}{h}\int_{F} d_E(x)\,dx\right\}
{ -\frac{1}{h}\int_E d_E(x)\, dx,
}
\end{multline}
where $d_E(x)=\dist(x,E)-\dist(x,\R^d\setminus E)$. The existence of a solution to~\eqref{NLATW}
is not totally obvious, however, it can be established by considering the equivalent
convex variational problem
\[
\min_{u\in L^1(\R^d;[0,1])}\left\{ \E^f(u)\ +\ \frac{1}{h}\int_{\R^d} u(x) d_E(x)\,dx\right\}
\]
with $\E^f$ defined in~\eqref{coareaEf}, and observing that, given a solution of that
problem, for a.e. $s\in (0,1)$ the sets $\{u>s\}$ and $\{u\ge s\}$ are a solution to~\eqref{NLATW}.
The existence of a minimal (or maximal) solution follows from the fact that if $E$, $E'$ are solutions,
then also $E\cap E'$ and $E\cup E'$ are, thanks to~\eqref{submod}. Moreover, it is not difficult
to see that if $F$ solves \eqref{NLATW}, then
$$
\M^f(F)=\E^f(F)\,,
$$
{ where $\M^f(F)$ is defined in  \eqref{varmr}.
}
The following classical lemmas hold.
%\begin{lemma}\label{lemcomp}
%Let $E\subseteq E'$, so that $d_E\ge d_{E'}$. Then $T^\pm_h E\subseteq T^\pm_h E'$. Moreover
%if $E\subset\subset E'$, $T^+_h E\subseteq T^-_h E'$.
%\end{lemma}
{
\begin{lemma}\label{lemcomp}
If $E\subset\subset E'$, then $T^+_h E\subseteq T^-_h E'$.
Moreover, if $E\subseteq E'$, then $T^\pm_h E\subseteq T^\pm_h E'$. 
\end{lemma}
}
\begin{proof}
The proof is classical and we just sketch it. %% (see~\cite{...}).  WHAT?
We  first assume that $E\subset\subset E'$, so that $d_E> d_{E'}$ a.e.
We compare  the energy~\eqref{NLATW} of $F=T^+_h E$ with the one of $F\cap F'$,  
where $F'=T^-_h E'$,
and the energy { \eqref{NLATW} (with $E$ replaced by $E'$)
of $F'$ with the one of $F\cup F'$. 
}
We sum both inequalities and use~\eqref{submod} to deduce that $F\subseteq F'$.

Now, if $d_E\ge d_{E'}$, we replace $d_E$ with $d_E+\e$ and observe that the corresponding
minimal solutions $F_\e$ and $F'$  satisfy $F_\e \subseteq F'$. 
{
Let $F_0$ be the Kuratowsky limit of $F_\e$ (up to a subsequence). Then, it is easy to see that $F_0$ is a solution, and 
$T_h^- E\subseteq F_0\subseteq F'=T_h^- E'$.
}
The proof for $T_h^+$ is almost identical.
\end{proof}
\begin{lemma}\label{lemdist}
Let $E\subset\subset E'$ and let  $\delta=\dist(\partial E,\partial E') >0$. Then 
$T^+_h E \subset\subset T^-_h E'$ and,  more precisely, 
$\dist(\partial T^+_hE,\partial T^-_h E')\ge \delta$.
\end{lemma}
\begin{proof}
Let $z\in \R^d$ with $|z|<\delta$: $z+E\subset\subset E'$ so that
$T^+_h (z+E) \subseteq T^-_h(E)$. By translation invariance of the scheme it follows that
$z+T^+_h(E)\subseteq T^-_h(E)$, and we deduce the thesis.
\end{proof}

If $E$ is a non-compact set with compact boundary, we can define $T^\pm_h E$ in a similar way
(or simply let $T^\pm_h E = \R^d\setminus (T^\mp_h(\R^d\setminus E))$), and still the comparison
holds.
Thanks to the comparison lemma~\ref{lemcomp}, starting from a function $u\in BUC_c(\R^d)$ (with compact support, or constant outside of a compact set),
for $s>s'$  we have 
 $T^+_h\{u\ge s\} \subseteq  T^-_h\{u\ge s'\}$. It follows that we can
define a function
\[
T_h u(x) \ :=\ \sup\{ s\ :\ x\in T^+_h \{u\ge s\}\}\ =\ \sup\{ s\ :\ x\in T^-_h \{u\ge s\}\}.
\]
We easily see that for a.e.~$s$, $\{ T_h u \ge s\}= T^\pm_h \{u\ge s\}$.
Using Lemma~\ref{lemdist}, we find that the distance between two such level sets of $T_h u$
is larger than the distance between the corresponding level sets of $u$: hence 
$T_h u \in BUC_c(\R^d)$, with the same modulus of continuity. Finally, we can deduce (by
approximation) that for any level $s\in \R$, 
\[
T_h^- \{u\ge s\}\ =\ \{ T_h u > s\}\ ,\qquad T_h^+ \{u\ge s\}\ =\ \{ T_h u \ge s\}.
\]

Now, starting from $u_0$, we build a function $u_h(x,t):\R^d\times \R_+\to \R$ by letting
\[
u_h(x,t)\ :=\ (T_h)^{[\frac{t}{h}]} u_0
\]
where $[\cdot]$ is the integer part. By construction, $u_h$ has a uniform spatial modulus
of continuity. 
The next lemma deals with the non-local evolution of balls.
\begin{lemma}\label{balls}
Let $x\in\R^d$, $r_0>0$ and let $E_0=B(x,r_0)$. Then for every $h,t>0$ we have 
$$
({T^\pm_h})^{[\frac{t}{h}]}(E_0) =  B(x, r_h^\pm(t)),
$$
for some $r_h^\pm(t) \ge 0 $. Moreover,
$
r^\pm_h(t) \to r(t)
$
uniformly in $[0,T^*(r_0)]$ as $h\to 0$, where  $r$ is the solution to 
\begin{equation}\label{ode}
\left\{
\begin{array}{ll}
\displaystyle \dot r(t) = \int_{\bar \delta}^\delta f'(s) \left[(1+\frac{s}{r})^{d-1} - \big( (1-\frac{s}{r})^+\big)^{d-1} \right] ds,\\
r(0) = r_0, 
\end{array}
\right.
\end{equation}
and  $T^*(r_0)$ is extinction time of $r(t)$ (i.e, such that $r(T^*(r_0))=0$).
Finally, there exists $c_0>0$ such that for every $r_0 \le1$ we have 
\begin{equation}\label{czero}
T^*(r_0)\ge c_0\, r_0^d.
\end{equation}
\end{lemma}
\begin{proof}
By translation invariance we may assume $x=0$. Since the union of any family of minimizers of \eqref{NLATW} is still a minimizer, we deduce that any rotation of  $({T^+_h})^{[\frac{t}{h}]}(E_0)$ is contained in  
${(T^+_h)}^{[\frac{t}{h}]}(E_0)$ i.e., the maximal solution is radially symmetric. Analogously, by the stability of the minimality property with respect to intersection we deduce that   $(T^-_h)^{[\frac{t}{h}]}(E_0)$ is radially symmetric.
By a rearrangement procedure it can be readily seen that
 the maximal and minimal solutions are in fact balls.  Indeed, let $r\ge 0$  be determined by $|B(r)| = {(T^+_h)}^{[\frac{t}{h}]}(E_0)$. Then it is easy to see
that
\begin{align*}
& \E^f(B(r)) \le \E^f((T^\pm_h)^{[\frac{t}{h}]}(E_0))\,, \\
& \int_{B(r) \triangle E_0} \dist(x,\partial E_0)\,dx \le 
\int_{ (T^\pm_h)^{[\frac{t}{h}]}(E_0) \triangle E_0} \dist(x,\partial E_0)\,dx\,
\end{align*}
with strict inequality whenever the radially symmetric set $(T^\pm_h)^{[\frac{t}{h}]}(E_0)$ is not a ball.  

For $0<r<R$ let $e(r,R)$ be the total energy in \eqref{NLATW} for $E= B_R$ and $F= B(r)$, i.e.,  
{
\begin{equation}
e(r,R)\ =\ - \int_{\underline\delta}^\delta f'(s) \omega_d [(r+s)^d - ((r-s)^+)^d]\,ds
\,+\, \frac{d\omega_d}{h}\int_r^R (R-s)s^{d-1}\,ds,
\end{equation}
where $\omega_d$ denotes, as usual, the volume of the unit ball in $\R^d$.
}
A straightforward computation shows that  $\frac{\partial}{\partial r} e(r,R) = 0$  is equivalent to
\begin{equation}\label{euleropalle}
\frac{1}{h} (r-R) = \int_{\underline\delta}^\delta f'(s) [(1+\frac{s}{r})^{d-1} - ((1-\frac{s}{r})^+)^{d-1}] ds.
\end{equation}
Now we construct the approximated evolution starting from $B(r_0)$. To this purpose, let us set $r_{h,0} = r_0$ and define 
$ r_{h,i} $ recursively,  as the minimum point of $e(r, r_{h,i-1})$ (and we stop if $ r_{h,i} =0$). Denote by $\hat r_h(t)$ the piecewise affine interpolation of $r_{i,h}$ given by
$$
\hat r_h(t)=r_{h, [\frac th]}+\left(t-\left[\tfrac th\right]\right)\left(r_{h, [\frac {t+1}h]}-r_{h, [\frac th]}\right)\,.
$$
 Then, by \eqref{euleropalle},  $\hat r_h(t)$ solves 
$$
\left\{
\begin{array}{ll}
\displaystyle \frac{d}{dt} \hat r_h(t) = g\left(\hat r_h \left(\left[\tfrac{t+1}{h}\right] \right) \right)  ;\\
\hat r_h(0) = r_0,
\end{array}
\right.
$$
where
$$
g(r):= \int_{\bar \delta}^\delta f'(s) \left[(1+\frac{s}{r})^{d-1} - \big( (1-\frac{s}{r})^+\big)^{d-1} \right] ds.
$$
Let $[0,T^*(r_0))$ be the maximal interval of  definition for the solution to problem \eqref{ode}. Clearly, we have $r(T^*(r_0))=0$. Moreover,  standards stability arguments in ODE yield that $\hat r_h$, and in turn $r_h$, converge uniformly to $r$ in $[0,T]$ for every $T<T^*(r_0)$. The uniform convergence in  $[0,T^*(r_0)]$ follows by monotonicity.

Noticing that, for $r\le1$, $|g(r)| \le c \, r^{1-d}$ for some $c>0$, the final bound on $T^*(r_0)$  follows by  comparing with the solution to 
%\begin{equation}\label{estimradiusballs}
$$
\left\{
\begin{array}{ll}
\displaystyle \dot r(t) = -c r^{1-d},\\
r(0) = r_0. 
\end{array}
\right.
$$
%\end{equation}
\end{proof}

\begin{lemma}\label{timecontinuity}
There exists a time modulus of continuity $\hat\omega$, such that for any 
$\delta>0$, there exists $h(\delta)$ such that
if $x\in\R^d$, $0<h\le h(\delta)$, and
$t,s\ge 0$ with $|t-s|\le\delta$ then
\[
|u_h(x,t)-u_h(x,s)|\ \le\ \hat\omega(\delta)
\]
\end{lemma}
\begin{proof}
{
Let $\omega$ be a spatial
modulus of continuity for $u_0$, and  therefore  also for $u_h(\cdot, t)$ with $t\ge 0$. 
%That is, for any $x,y,t$ we have $|u_h(x,t)-u_h(y,t)| \le \omega(|x-y|)$.
Fix $r_0>0$. Then,  $u_h(y,t)\le u_h(x,t)+\omega(r_0)$ for all
$y\in B_{r_0}(x)$. Lemma~\ref{balls} shows that if $h$ is small enough,
then $u_h(x,t+s)\le u_h(x,t)+\omega(r_0)$ for $s\le c_0 r_0^d/2$.
Analogously, by $u_h(y,t)\ge u_h(x,t)- \omega(r_0)$ for all
$y\in B_{r_0}(x)$ we deduce $u_h(x,t+s)\ge u_h(x,t)-\omega(r_0)$ for $s\le c_0 r_0^d/2$.
The thesis follows if we choose $r_0= (2\delta/c_0)^{1/d}$, 
$\hat\omega(\delta)=\omega(r_0)$.
}
% Given $R>0$ and $B_R=\{x\,:\, |x|\le R\}$,
%  it is easy to check that $T^\pm_h B_R$
% must also be balls, of radius $S^\pm_h(R)$. We need to estimate this quantity.
% It should minimize the energy
% \[
% e(r)\ =\ \int_0^\delta\frac{\alpha(s)}{2s} ((r+s)^{d} - ((r-s)^+)^d)\,ds
% \,+\, \frac{d}{h}\int_r^R (R-s)s^{d-1}\,ds
% \]
% then we need to deduce that for any $R>0$, there exists $h(R),t(R)>0$ such that
% if $h\le h(R)$, $t\le t(R)$ then $(S_h)^{[t/h]}(R) \ge R/2$. We deduce that for any
% $\e>0$, choosing $R=\omega^{-1}(\e)$, if $h\le h(R)$ and $s\le t(R)$,
% then $|u_h(x,t+s)-u_h(x,t)|\le \e$. 
% \problem{ Write this proof, compute this estimate! and write maybe
% more correctly the proposition [TO BE DETAILED...]}
\end{proof}

Thanks to Lemma~\ref{timecontinuity}, we can extract a subsequence $(h_k)_{k\ge 1}$
such that $u_{h_k}$ converges locally uniformly  in $\R^d\times \R_+$ to a function
$u(x,t)$ which is bounded and uniformly continuous in space and time. 

\begin{remark}\label{conflin}
{\rm
Let $h_n\to 0$ be such that $u_{h_n}$ admits a limit $u$. Then, as a straightforward consequence of Lemma \ref{balls}, we deduce that if for some level
$s\in\R$,
$u(t, \cdot)\ge s$ (resp., $\le s$) on a ball of radius $r_0$,
then $u(\cdot,t')\ge s$ (resp., $\le s$) on  the concentric ball  with radius $r(t'-t)$, for $t'\ge t$, provided that $r(t'-t)>0$ (here $r(\cdot)$ solves~\eqref{ode}).
}
 \end{remark}

We can now show the main result of this section.

\begin{theorem}\label{thexist}
The limit $u$ is a viscosity solution of~\eqref{levelsetf}, in the
sense of Definition~\ref{defvisco}.
\end{theorem}

\begin{proof} %%%[Proof of Theorem~\ref{thexist}]

First it is clear, by construction, that $u(0,\cdot)=u_0$. Hence we need to show that the
equation holds for $t>0$.
{
 We only prove that it is  a subsolution,  the proof that it is a supersolution being  identical.
} 
Let $\f\in C^\infty(\R^d\times \R_+)$ and
$(\bar x, \bar t)\in \R^d\times \R_+$ be a maximum point of $u-\f$.
We may assume that 
this is a strict maximum point { and that $\f$ is coercive
}
: if not, we should first replace
(as usual) $\f$ with $\f(x,t)+\eta(|x-\bar x|^2+|t-\bar t|^2)$, derive
an inequality for this modified function, and send $\eta\to 0$,
which will give the desired inequality thanks to \eqref{Ffsci}.

By standard methods, we can then find $(x_k, t_k)\to (\bar x,\bar t)$
such that $t_k>0$ and $u_{h_k}-\f$ has a 
{
maximum 
}
at $(x_k, t_k)$. 
%%%% observe somewhere that $D$ is the spatial derivative (in the beginning)
\smallskip

\noindent\textbf{Step 1.}
Let us first assume that $D\f(\bar x, \bar t)\neq 0$
so that in particular, for $k$ large enough, $D\f(x_k, t_k)\neq 0$. 
We have that for all $(x,t)$,
{
\begin{equation}\label{ineqk}
u_{h_k}(x,t)\ \le\ \f(x,t) \,+c_k 
\end{equation}
where $c_k:=
[u_{h_k}(x_k, t_k)-\f(x_k, t_k)]$, 
} 
with equality if $(x,t)=(x_k, t_k)$. Let $\eta>0$ and 
{ $\f^{\eta}_{h_k}
:\R^d\to \R$
}
given by
\[
\f^{\eta}_{h_k}(x)\ =\ \f(t_k,x) \,+\, c_k\,+
\, \frac{\eta}{2}|x-x_k|^2\,,
\]
 then, for all $x\in \R^d$,
\[
u_{h_k}(t_k,x)\ \le\ \f^{\eta}_{h_k}(x)
\]
with equality \textit{if and only if} $x=x_k$.
Let $\e>0$ and consider the open, nonempty set
$V_\e=\{x\,:\, u_{h_k}(t_k,x)> \f^{\eta}_{h_k}(x)-\e\}$, which has positive measure,
contains $x_k$, and converges to $\{x_k\}$ in the Hausdorff sense as
$\e\to 0$.
{
In particular, setting  $s_\e:= u_{h_k}(x_k, t_k) -  \e/2$,
we have that  for  $\e>0$ sufficiently small  $|W_\e|>0$, where
\[
W_\e \ :=\ \{x\in\R^d\,:\, u_{h_k}(t_k,x)\ge s_\e\}
\setminus \{x\in\R^d\,:\, \f^{\eta}_{h_k}(x)\ge \e+s_\e\}\ \subseteq V_\e\,,
\]
}
Now, by minimality, we have
\begin{multline*}
\E^f(\{ u_{h_k}(\cdot, t_k)\ge s_\e\})\,+\,\frac{1}{h_k}
\int_{\{ u_{h_k}(\cdot, t_k)\ge s_\e\}} d_{\{ u_{h_k}(\cdot, t_k-h_k)\ge s_\e\}}(x)\,dx
\\ \le\,
\E^f(\{ u_{h_k}(\cdot, t_k)\ge s_\e\}\cap\{\f^{\eta}_{h_k}\ge \e+s_\e\})
\\ +\,\frac{1}{h_k}
\int_{\{ u_{h_k}(\cdot, t_k)\ge s_\e\}\cap\{\f^{\eta}_{h_k}\ge \e+s_\e\}}
 d_{\{ u_{h_k}(\cdot, t_k-h_k)\ge s_\e\}}(x)\,dx\,.
\end{multline*}
Adding to both sides the term
$\E^f(\{ u_{h_k}(\cdot, t_k)\ge s_\e\}\cup\{\f^{\eta}_{h_k}\ge \e+s_\e\})$
and using \eqref{submod}, we obtain
%\begin{multline*}
\[
\E^f(\{\f^{\eta}_{h_k}\ge \e+s_\e\}\cup W_\e)-\E^f(\{\f^{\eta}_{h_k}\ge \e+s_\e\})
\\+\,\frac{1}{h_k}\int_{W_\e}  d_{\{ u_{h_k}(\cdot, t_k-h_k)\ge s_\e\}}(x)\,dx \ \le\ 0\,.
\]
%\end{multline*}
{
Observing that by \eqref{ineqk}, $\{ u_{h_k}(\cdot, t_k-h_k)\ge s_\e\}
\subseteq  \{\f(\cdot, t_k-h_k) \ge s_\e- c_k\}$,  we also have
}
%\begin{multline}
\begin{equation}\label{finalineq}
\E^f(\{\f^{\eta}_{h_k}\ge \e+s_\e\}\cup W_\e)-\E^f(\{\f^{\eta}_{h_k}\ge \e+s_\e\})
\\+\,\frac{1}{h_k}\int_{W_\e}  d_{\{\f(\cdot, t_k-h_k) \ge s_\e-c_k\}}(x)\,dx \ \le\ 0\,.
\end{equation}
% \end{multline}

Now notice that for $z\in W_\e$ we have
\begin{equation}\label{E1}
s_\e\ \le\ \f(z, t_k) + c_k + \frac{\eta}{2}|z-x_k|^2  \ <\ \e+s_\e.
\end{equation}
{
In particular,  
\begin{equation}\label{inparticular}
W_\e \subseteq B_{C\sqrt{\e}}(x_k).
\end{equation} 
Moreover, 
}
\begin{equation}\label{E2}
\f(z, t_k-h_k) = \f(z, t_k)-h_k\partial_t \f(z, t_k) +
h_k^2 \int_0^1(1-s) \partial_{tt}^2 \f(z, t_k-sh_k) \,ds\,.
\end{equation}
If $y$ is the point closest to $z$ with $\f(y, t_k-h_k)=s_\e-c_k$,
so that $|y-z|=|d_{\{\f (\cdot, t_k-h_k)\ge s_\e-c_k\}}(z)|$, then
\begin{multline}\label{E3}
\f(z, t_k-h_k)\,=\, \f(y, t_k-h_k)+(z-y)\cdot D\f(y, t_k-h_k)
\\+
\int_0^1(1-s) (D^2\f(y+s(z-y), t_k-h_k)(z-y))\cdot(z-y)\,ds
\\=\,s_\e-c_k  - d_{\{\f (\cdot, t_k-h_k)\ge s_\e-c_k\}}(z)| D\f(y, t_k-h_k)|
\\+
\int_0^1(1-s) (D^2\f(y+s(z-y), t_k-h_k)(z-y))\cdot(z-y)\,ds\,.
\end{multline}
Combining \eqref{E1}, \eqref{E2}, and \eqref{E3}, we deduce 
\begin{multline*}
d_{\{\f(\cdot, t_k-h_k)\ge s_\e-c_k\}}(z)| D\f(y, t_k-h_k)|
\\ \ge\ -\e  + h_k\partial_t \f(z, t_k) \,-\,
h_k^2 \int_0^1(1-s) \partial_{tt}^2 \f(z, t_k-sh_k) \,ds\\ +\,
\int_0^1(1-s) (D^2\f(y+s(z-y), t_k-h_k)(z-y))\cdot(z-y)\,ds\,.
\end{multline*}
{
Note that, in view of \eqref{E1},  $|\varphi( z, t_k)-\varphi( y, t_k)|\leq \e+Ch_k=O(h_k)$,
} 
provided that
$\e<< h_k$ are small enough.
In turn, as $|D\f(x_k, t_k)|\neq 0$, we  have 
 $|z-y|=O(h_k)$ and,
 {
 using also \eqref{inparticular}, 
 } 
 we deduce
\begin{multline}
\label{ineqdist}
\frac{1}{h_k}d_{\{\f(\cdot, t_k-h_k)\ge s_\e-c_k\}}(z) \, \ge\,
\frac{\partial_t \f(z, t_k) -\frac{\e}{h_k}  +  O(h_k)}{| D\f(y, t_k-h_k)|}
\\=\,\frac{\partial_t \f(x_k, t_k) +O(\sqrt{\e}) -\frac{\e}{h_k}  +  O(h_k)}
{| D\f(x_k, t_k)| + O(h_k)}\,.
\end{multline}
%Assuming that $\e=O(h_k^2)$, it follows that
%\begin{equation}\label{ineqdist}
%\frac{1}{h_k}\int_{W_\e}  d_{\{\f(\cdot, t_k-h_k) \ge s_\e-c_k\}}(x)\,dx \ \ge\ 
%|W_\e| \left(\frac{\partial_t \f(x_k, t_k)}
%{| D\f(x_k, t_k)|} + O(h_k)\right)\,.
%\end{equation}
%% BUT now the O(h_k) depends on \eta, this is stupid!! the previous 
%% inequality is better

We now focus on the term
\[
\E^f(\{\f^{\eta}_{h_k}\ge \e+s_\e\}\cup W_\e)-\E^f(\{\f^{\eta}_{h_k}\ge \e+s_\e\})
\]
of inequality~\eqref{finalineq}. 
{
This is the sum of the two following
expressions, which we will estimate separately:
}
\begin{equation}\label{pluss}
\int_0^\delta-f'(s)\big(|(\{\f^{\eta}_{h_k}\ge \e+s_\e\}\cup W_\e)+B_s|
-|\{\f^{\eta}_{h_k}\ge \e+s_\e\}+B_s|\big)\,ds\,,
\end{equation}
\begin{equation}\label{minuss}
\int_0^\delta -f'(s)\big(|\{\f^{\eta}_{h_k}\ge \e+s_\e\}\ominus B_s|
-|(\{\f^{\eta}_{h_k}\ge \e+s_\e\}\cup W_\e)\ominus B_s|\big)\,ds\,,
\end{equation}
where $A\ominus B$ denotes the set $\{ x\,:\, x+B\subseteq A\}$.
We recall that by assumption, $f'(s)= 0$  for
$s\le \underline{\delta}$, so that the integrals are in fact
on $[\underline{\delta},\delta]$.
% We consider $\sqrt{\e}<< \underline{\delta}$
%(this will be made precise later on). USELESS

Let us first consider~\eqref{pluss}.
For any $x$ in a neighborhood of $x_k$, we have 
$x\in \partial \{\f^{\eta}_{h_k}\ge \f^{\eta}_{h_k}(x)\}$ and we can define
$s^*(x)\in (0,\delta]$ such that for $s\in (0,s^*(x)]$,
$\dist(x+ s\nu(x),\{\f^{\eta}_{h_k}\ge \f^{\eta}_{h_k}(x)\})= s$ and
for $s\in (s^*(x),\delta]$ (possibly empty), 
$\dist(x+ s\nu(x),\{\f^{\eta}_{h_k}\ge \f^{\eta}_{h_k}(x)\})<s$.
Here $\nu(x)= -D\f^{\eta}_{h_k}(x)/|D\f^{\eta}_{h_k}(x)|$, and it is important
to observe that thanks to the regularity of $\f^{\eta}_{h_k}$, $s^*(x)$
is continuous near $x_k$.

If $\e$ is small, for $x\in \partial\{\f^{\eta}_{h_k}\ge \e+s_\e\}$, there exists
a minimal $\overline{h}^\e(x)\ge 0$, with $\overline{h}^\e(x)\le C\sqrt{\e}$,   such that
$W_\e\cap\{ x+t\nu(x), t\in [0,\delta]\}\subseteq
\{ x+t\nu(x), t\in [0,\overline{h}^\e(x)]\}$. %% FAIRE UN DESSIN
Clearly, for $s\geq \underline\delta$ and $\e$ small enough, 
\begin{multline*}
\left(\{\f^{\eta}_{h_k}\ge \e+s_\e\}\cup W_\e\,+\,B_s\right)
\setminus \left(\{\f^{\eta}_{h_k}\ge \e+s_\e\}\,+\,B_s\right)
 \\ \supseteq \ 
\big\{ x+t\nu(x)\,:\, x\in \partial \{\f^{\eta}_{h_k}\ge \e+s_\e\}\,,
s\le t\le \min\{s^*(x),s+\overline{h}^\e(x)\}\big\}
\end{multline*}
The volume of this latter set is
\[
\int_{\partial\{\f^{\eta}_{h_k}\ge \e+s_\e\}} 
\int_{I(x)} \det(I+t\nabla \nu(x))\,dt \,d\H^{d-1}(x)
\]
where $I(x)$ is the interval (possibly empty)
$\{t\,:\,s\le t\le \min\{s^*(x),s+\overline{h}^\e(x)\}$.
{
Fix $\sigma>0$. A simple continuity argument yields that, 
for  $\e$  sufficiently small, 
if $\underline{\delta}\le s \le s^*(x_k)-\sigma$,
then $s+\overline{h}^\e(x)\le s^*(x)$.
} 
We deduce that
\begin{multline}\label{longandterrible}
|(\{\f^{\eta}_{h_k}\ge \e+s_\e\}\cup W_\e\,+\,B_s) \setminus (\{\f^{\eta}_{h_k}\ge \e+s_\e\}\,+\,B_s)|
\\\ge\, \int_{\{ \f^{\eta}_{h_k}=\e+s_\e ,\ \overline{h}^\e>0\}}
\int_s^{s+\overline{h}^\e(x)} \det(I+t\nabla \nu(x))\,dt \,d\H^{d-1}(x)
\\
\,=\, \int_{\{ \f^{\eta}_{h_k}=\e+s_\e ,\ \overline{h}^\e>0\}} \det(I+s\nabla \nu(x))
\int_0^{\overline{h}^\e(x)} \frac{\det(I+(s+t)\nabla \nu(x))}{\det(I+s\nabla \nu(x))}
\,dt\,d\H^{d-1}(x)
\\
\,\ge \, (\det(I+s\nabla \nu(x_k))+O(\sqrt{\e}))
(1+O(\sqrt{\e}))|W_\e|\,,
\end{multline}
where we have used that
\[
|W_\e| \ \le\ \int_{\{ \f^{\eta}_{h_k}=\e+s_\e ,\ \overline{h}^\e>0\}} 
\int_0^{\overline{h}^\e(x)} \det(I+t\nabla \nu(x))\,dt \,d\H^{d-1}(x)\,.
\]
Finally, integrating over $s\in [\underline{\delta},s^*(x^k)-\sigma]$
we deduce that \eqref{pluss} is larger than 
\begin{multline}\label{estimpluss}
(1+O(\sqrt{\e}))|W_\e|
\int_{\underline{\delta}}^{s^*(x^k)-\sigma} -(2sf'(s))(\kappa^+_s(K^\eta_k,x_k)+O(\sqrt{\e}))\,ds
\\ =\ 
|W_\e|(\kappa^+_f(K^\eta_k,x_k) + O(\sigma))\,,
\end{multline}
where the ``curvature'' (defined in~\eqref{kappaprho} and~\eqref{kappapmf})
is relative to the set $K^\eta_k=\{\f^{\eta}_{h_k}\ge \f^{\eta}_{h_k}(x_k)\}$.
Here we have used the fact that $\sqrt{\e}<\sigma$, so that
$O(\sqrt{\e})=O(\sigma)$.
\smallskip

Now let us estimate the negative quantity \eqref{minuss}. It is very
similar, but not equivalent. 

We now need to introduce the function $\underline{h}^\e$, defined for
$x\in \partial\{\f^{\eta}_{h_k}\ge \e+s_\e\}$, which is the largest real number
(which can be nonzero only in a neighborhood of $x_k$) such that
$x\,+\, s \nu(x)\in W_\e$ for all $s\in (0,\underline{h}^\e(x))$.
For $x$ in a neighborhood of $x_k$,
we also introduce $s_*(x)\in (0,\delta]$, such that
$\dist(x-s\nu(x),\partial \{ \f^{\eta}_{h_k}\ge \f^{\eta}_{h_k}(x)\} ) = s$ for $s\le s_*(x)$
and
$\dist(x-s\nu(x),\partial \{ \f^{\eta}_{h_k}\ge \f^{\eta}_{h_k}(x)\} ) < s$ for
$s_*(x)<s\le \delta$. Just as $s^*$, this quantity is continuous
with respect to $x$, thanks to the regularity of $\f^{\eta}_{h_k}$.
If $x\in \partial\{\f^{\eta}_{h_k}\ge \e+s_\e\}$ and $s\le s_*(x)$,
one has that
\begin{multline*}
\bigcup_{\{ \f^{\eta}_{h_k}=\e+s_\e ,\ \underline{h}^\e>0\}}
\big\{ x+t\nu(x)\,:\, -s\le t\le -s+\underline{h}^\e(x)\big\}
\\ \supseteq\ 
\big((\{\f^{\eta}_{h_k}\ge \e+s_\e\}\cup W_\e)\ominus B_s\big)\setminus
\big(\{\f^{\eta}_{h_k}\ge \e+s_\e\}\ominus B_s\big)\,,
\end{multline*}
at least when $W_\e$ is small enough with respect to the scale $s$ (which
can be ensured, as we need to consider only $s\ge\underline{\delta}$).
It follows that
\begin{multline}\label{longandmoreterrible}
\left|\big((\{\f^{\eta}_{h_k}\ge \e+s_\e\}\cup W_\e)\ominus B_s\big)\setminus
\big(\{\f^{\eta}_{h_k}\ge \e+s_\e\}\ominus B_s\big)\right|
\\\le\, \int_{\{ \f^{\eta}_{h_k}=\e+s_\e ,\ \underline{h}^\e>0\}}
\int_{-s}^{-s+\underline{h}^\e(x)} \det(I+t\nabla \nu(x))\,dt \,d\H^{d-1}(x)
\\
\,=\, \int_{\{ \f^{\eta}_{h_k}=\e+s_\e ,\ 
{ 
\underline{h}^\e
}
>0\}} \det(I-s\nabla \nu(x))
\int_0^{\underline{h}^\e(x)} \frac{\det(I-(s-t)\nabla \nu(x))}{\det(I-s\nabla \nu(x))}
\,dt\,d\H^{d-1}(x)
\\
\,\le \, (\det(I-s\nabla \nu(x_k))+O(\sqrt{\e}))
(1+O(\sqrt{\e}))|W_\e|\,,
\end{multline}
where we have used, this time, that
\[
|W_\e| \ \ge\ \int_{\{ \f^{\eta}_{h_k}=\e+s_\e ,\ \underline{h}^\e>0\}} 
\int_0^{\underline{h}^\e(x)} \det(I+t\nabla \nu(x))\,dt \,d\H^{d-1}(x)\,.
\]
{
Fix $\sigma>0$ as before. 
%If $s\le s_*(x_k)-\sigma$,
%then $s\le s_*(x)$ for $x$ close enough
%to $x_k$ so that if $\e$ is sufficiently small, then \eqref{longandmoreterrible}
%holds for $s$. On the other hand, 
If $s\ge s_*(x_k)+\sigma$,
}
then $s\ge s_*(x)+\sigma/2$ near $x_k$ and we see that 
\begin{equation}\label{shortbutneverthelessterrible}
\left|\big((\{\f^{\eta}_{h_k}\ge \e+s_\e\}\cup W_\e)\ominus B_s\big)\setminus
\big(\{\f^{\eta}_{h_k}\ge \e+s_\e\}\ominus B_s\big)\right| = 0\,,
\end{equation}
provided that
$\e$ is small enough.
%Now if $s\in (s_*(x_k)-\sigma,s_*(x_k)+\sigma)$, we easily check that
%\eqref{longandmoreterrible} still holds. 
Integrating \eqref{longandmoreterrible}-\eqref{shortbutneverthelessterrible}
over $s\in [\underline{\delta},\delta]$
we deduce that \eqref{minuss} is larger than
\begin{multline}\label{estimminuss}
(1+O(\sqrt{\e}))|W_\e|
(\int_{\underline{\delta}}^{\delta} -(2sf'(s))(\kappa^-_s(K^\eta_k,x_k)+O(\sqrt{\e}))\,ds
+O(\sigma))
\\ =\ 
|W_\e|(\kappa^-_f(K^\eta_k,x_k) + O(\sigma))\,.
\end{multline}
It therefore follows from~\eqref{estimpluss} and \eqref{estimminuss} that
\[
\E^f(\{\f^{\eta}_{h_k}\ge \e+s_\e\}\cup W_\e)-\E^f(\{\f^{\eta}_{h_k}\ge \e+s_\e\})
\ \ge\ |W_\e|(\kappa_f(K^\eta_k,x_k) + O(\sigma))\,.
\]
Thanks to \eqref{ineqdist}, we deduce, after dividing \eqref{finalineq}
by $|W_\e|$ and sending $\e\to 0$ and $\sigma\to 0$, that
\begin{equation}\label{subsol1}
\frac{\partial_t \f (x_k, t_k)}{| D \f(x_k, t_k)|}
\ +\ \kappa_f(K^\eta_k,x_k)\ +\ O(h_k)\ \le\ 0
\end{equation}
{
where 
%$\kappa_f(K^\eta_k,x_k)$ is the non-local curvature of the level set
%$K^\eta_k=\{ \f^{\eta}_{h_k}\ge \f^{\eta}_{h_k}(x_k)\}$ and 
$O(h_k)$ depends only on the
regularity of $\f$.
}
We may therefore send $\eta$ to zero to
deduce that \eqref{subsol1} also holds also with
$K^\eta_k$ replaced by  $K_k=\{ \f(\cdot, t_k) \ge \f(x_k, t_k)\}$,
thanks to~\eqref{Ffsci} and~\eqref{HamilCurv}.
Letting now $k\to\infty$,  using~\eqref{Ffsci} 
{
and the monotonicity property iii), 
} we deduce
\begin{equation}\label{SUBSOL}
\partial_t \f(\bar x, \bar t)\ \,+\, 
F_f(\bar x, D\f(\bar x, \bar t),D^2\f(\bar x, \bar t),\{ \f(\bar t,\cdot)
\ge \f(\bar x, \bar t)\})\ \le\ 0\,,
\end{equation}
that is, $u$ is a viscosity subsolution at $(\bar t, \bar x)$.
\smallskip

\noindent\textbf{Step 2.}
Now we consider the case $D\f(\bar z)= 0$ and we  show that  $\f_t(\bar z) \le 0$. 
Let $\psi_n$ be defined as in \eqref{psienne}, with $T$ replaced by $\bar t$, and let 
$z_n=(x_n, t_n)$ be a sequence of 
{ 
maximizers 
}
of $u-\psi_n$, such that 
$x_n\to \bar x$ and $t_n\to \bar t^-$. If $x_n\neq \bar x$ for a (not relabeled) subsequence, then 
$D\psi_n(x_n, t_n)\neq 0$ and \eqref{SUBSOL} holds for $\psi_n$ at $z_n$. Passing to the limit and using the properties of $f$, we deduce that  $\f_t(\bar z) \le 0$ (see \eqref{boh} for
the details). 

We now assume that $z_n=(\bar x, t_n)$  for all $n$ sufficiently large.
Set $h_n:=\bar t-t_n$ and 
$$
r_n:=\sqrt[d]{\frac{2h_n}{c_0}}\,,
$$
where $c_0$ is the constant in \eqref{czero}. Note now that by \eqref{IS} and 
\eqref{cippi}, the function $f$ appearing in the definition of $\psi_n$ is of the form
$f(r)=g(r)r^d$, for a suitable function $g$ such that $g(r)\to 0^+$ as $r\to 0^+$.
It easily follows that
{
\begin{align*}
B(\bar x, r_n)& \subset \left\{ \psi_n(\cdot,t_n) \le \psi_n(\bar x, t_n) + f(r_n) \right\}\\
&= \left\{ \psi_n(\cdot,t_n) \le \psi_n(\bar x, t_n) + g(r_n)\frac{2h_n}{c_0} \right\}\\
&\subset
 \left\{ u(\cdot,t_n) \le u(\bar x, t_n) + g(r_n)\frac{2h_n}{c_0} \right\}\,,
\end{align*}
}
Note that the last inclusion follows from  the maximality 
of $u-\psi_n$ at $z_n$ and the fact that $u(z_n)=\psi_n(z_n)$.
By \eqref{czero}, the extinction time $T^*(r_n)$ of the ball $B(\bar x, r_n)$ under the non-local evolution satisfies $T^*(r_n)\geq c_0 r^d=2h_n$. Hence, by comparison 
(see Remark~\ref{conflin}), we deduce that
$$
\bar x\in \left\{ u(\cdot,\bar t) \le u(\bar x, t_n) + g(r_n)\frac{2h_n}{c_0} \right\}\,.
$$
Thus, using also the maximality of $u-\f$ at $\bar z$, 
$$
\frac{\f(\bar x, t_n) - \f(\bar z)}{-h_n}\leq 
\frac{u(\bar x, t_n) - u(\bar x, \bar t)}{-h_n}\leq g(r_n)\frac{2}{c_0}\,.
$$
Passing to the limit, we conclude that $\f_t(\bar z)\leq 0$.

\end{proof}

% \section{A time-discrete evolution scheme}

% In what follows $\Om=\R^d$.

% \subsection{}
% \[
% \min_E \E_\rho(E)\,+
% \frac{1}{h}\,\int_{E\triangle E^{n-1}} \dist(x,\partial E^{n-1})\,dx
% \]

% Comparison principle if we select the minimal or maximal solution.

% Introduce the level set formulation. Check that if $u_0$ is BUC,
% then $u_h(t)$ is for all time. Explicit solution for spheres.
% Continuity modulus in time.

% \subsection{the limit problem}
% In the limit, $u_h\to u$ locally uniformly up to a subsequence. 
% The issues are what equation is satisfied by $u$ and uniqueness.

% Formally, it is a geometric motion with $V=\kappa_\rho$, the
% non-local ``curvature'' associated to energy $\E_\rho$.

\section{Algorithm and numerical examples}\label{secnumer}
\begin{figure}[htb]
\centerline{
\includegraphics[height=5cm]{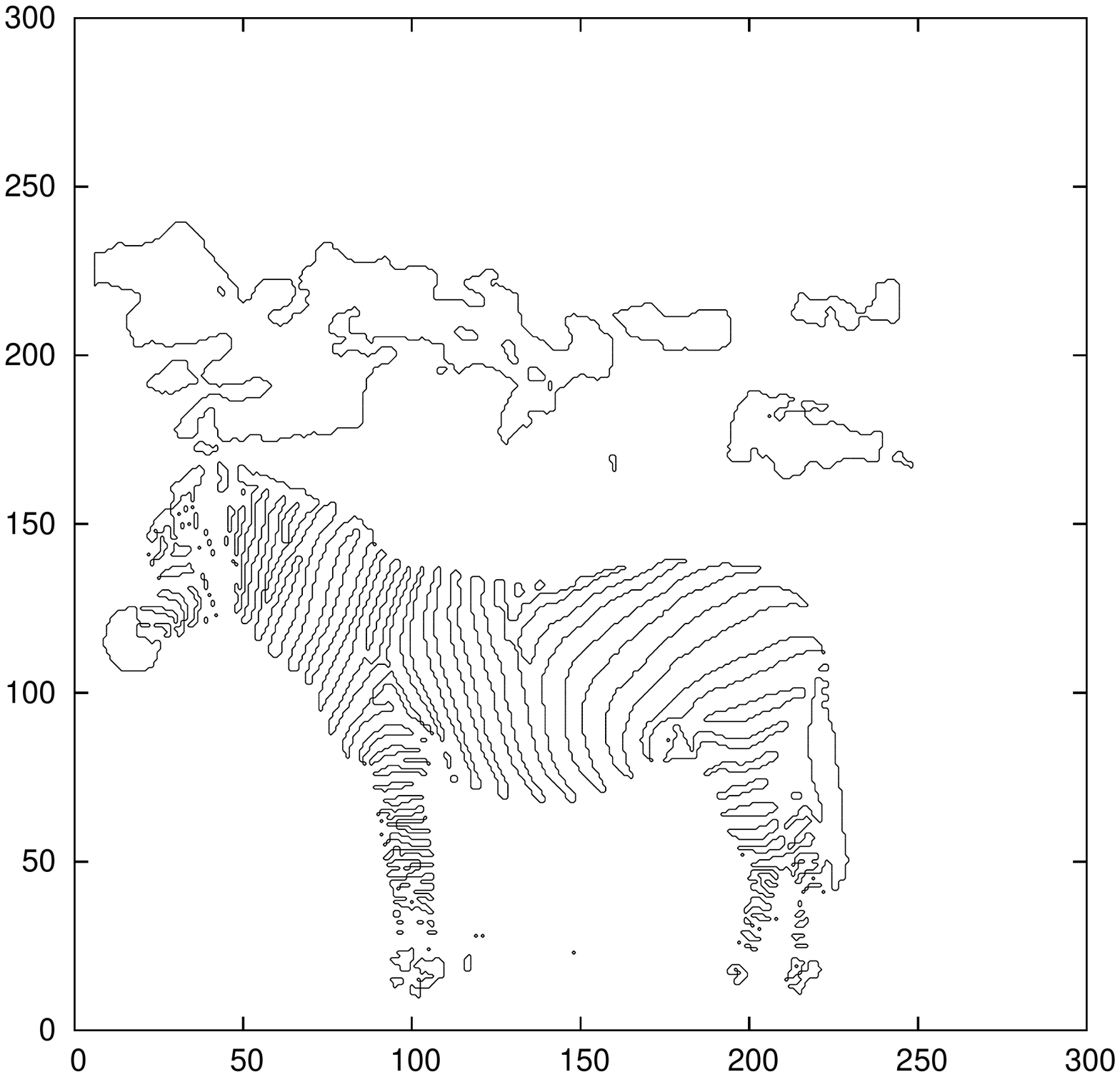}\hspace{-2.9cm}
\ \includegraphics[height=5cm]{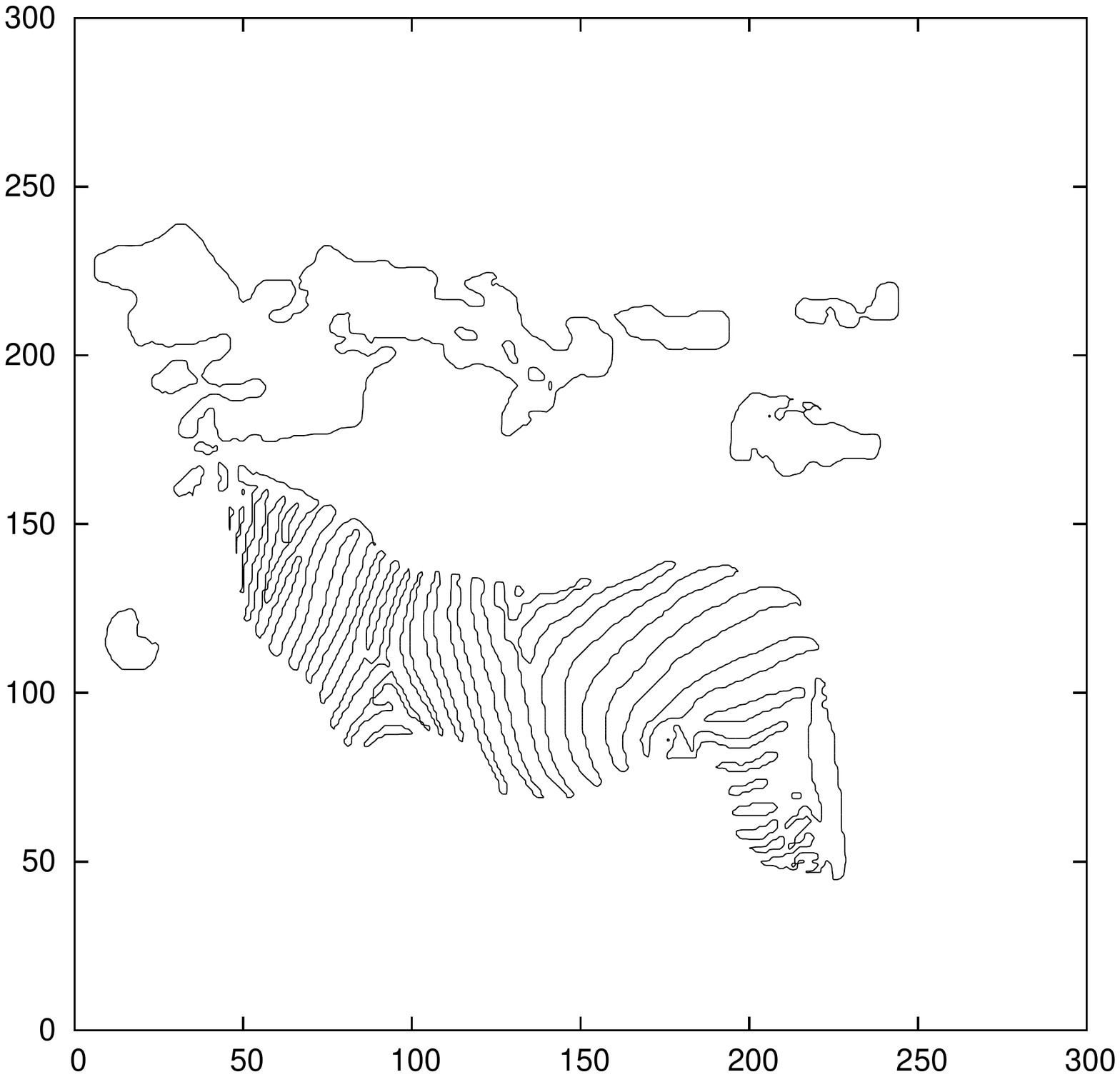}\hspace{-2.9cm}
\ \includegraphics[height=5cm]{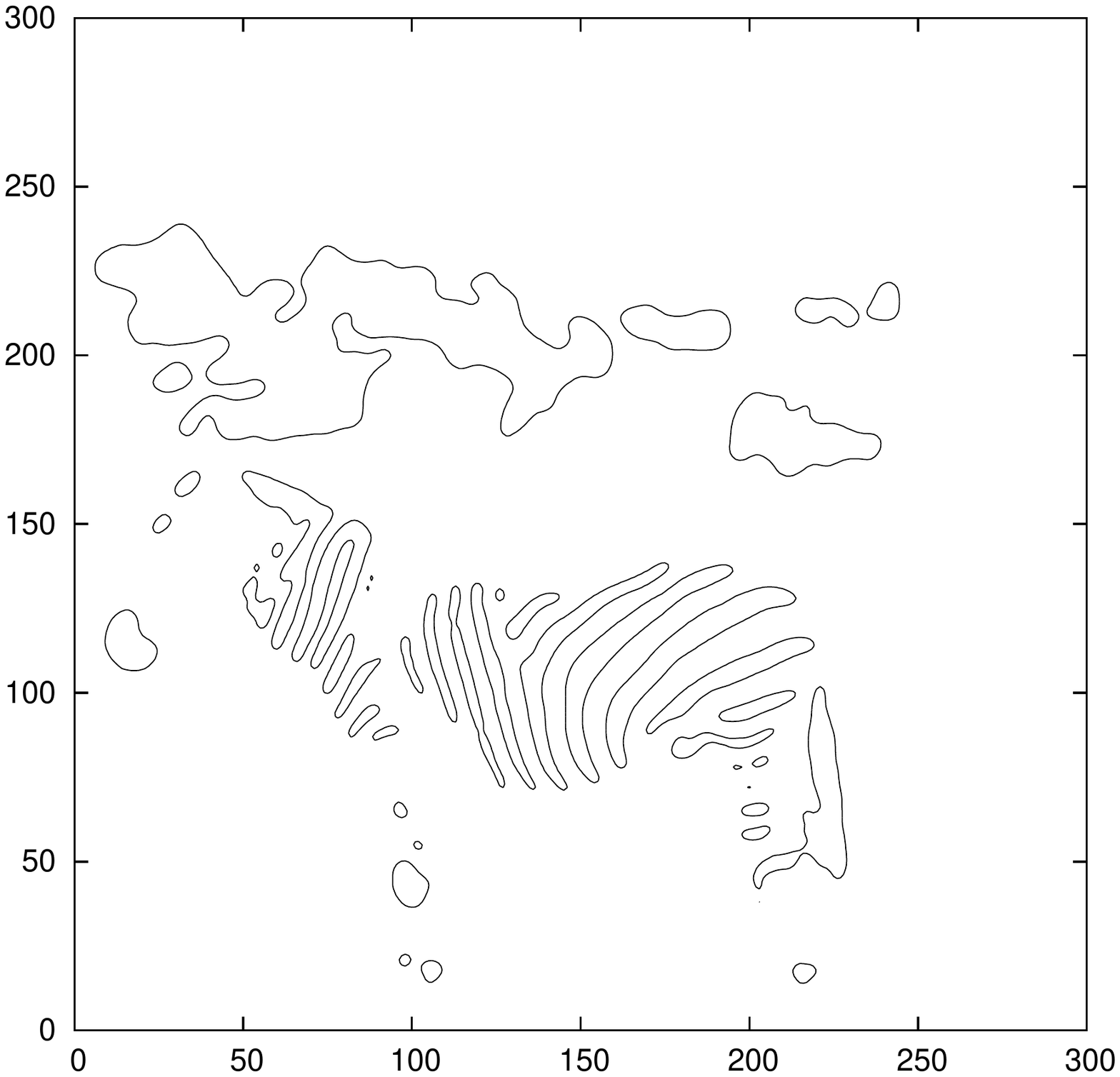}
}\caption{A zebra and its smoothing (left: starting image, center: the non-local motion,
right: the standard curvature flow), at a small time.}\label{figzebra1}
\end{figure}

\subsection{A numerical implementation of the time-discrete scheme}

We show in this section an example of evolution with the motion
studied in this paper, in dimension two.
Actually, the implementation is not straightforward
and only an approximate motion is computed, on a discrete rectangular
grid.
The approach we follow is described in~\cite{CDarbon}. It consists
in minimizing, given a discretization of the signed distance function
$d^{E^{n-1}}$ to the boundary of the set $E^{n-1}$ (negative inside,
positive outside), the energy
\begin{equation}\label{discATW}
\min_{u} J(u) \ +\ \frac{1}{2h} \|u-d^{E^{n-1}}\|^2
\end{equation}
and define $d^{E^n}$ as the signed distance function to
$\{u\le 0\}$, computed as precisely as possible using a
Fast-Marching algorithm~\cite{Tsitsiklis,SethianFM}. Here,
$u$, $d^{E^{n-1}}$ are defined on the discrete points $\{ (i,j)\,:\,
0\le i\le N-1\,, 0\le j\le M-1\}$, and the term
\[
 \|u-d^{E^{n-1}}\|^2\ =\ \sum_{i,j} (u_{i,j}-d^{E^{n-1}}_{i,j})^2
\]
is the Euclidean norm. A spatial discretization term can
be introduced in an obvious way. It turns out that if $J$ is a
correct approximation of the functional~\eqref{defminku}, then
the algorithm is an approximation of the time-discrete scheme~\eqref{NLATW}
studied in Section~\ref{secgeometric}. In this case, the iterations
should be an approximation of the motion driven by the energy.

The discretization of the ``total variation'' $J(u)$ is more complicated.
Actually, the simplest here is to approximate~\eqref{defminku}
rather than~\eqref{econeffu}. We fix $\rho>0$, Let $B$ be the
discrete ball $\{ (i,j)\in\Z^2\,:\, i^2+j^2\le \rho\}$, and
let
\[
J(u)\ =\ \frac{1}{2\rho} \sum_{i,j} \osc_{(i,j)+B} (u)
\]
where the oscillation (here simply the max minus the min)
is computed on the
finite sets $((i,j)+B)\cap [0,N-1]\times[0,M-1]$.
\begin{figure}[htb]
\centerline{
\includegraphics[height=5cm]{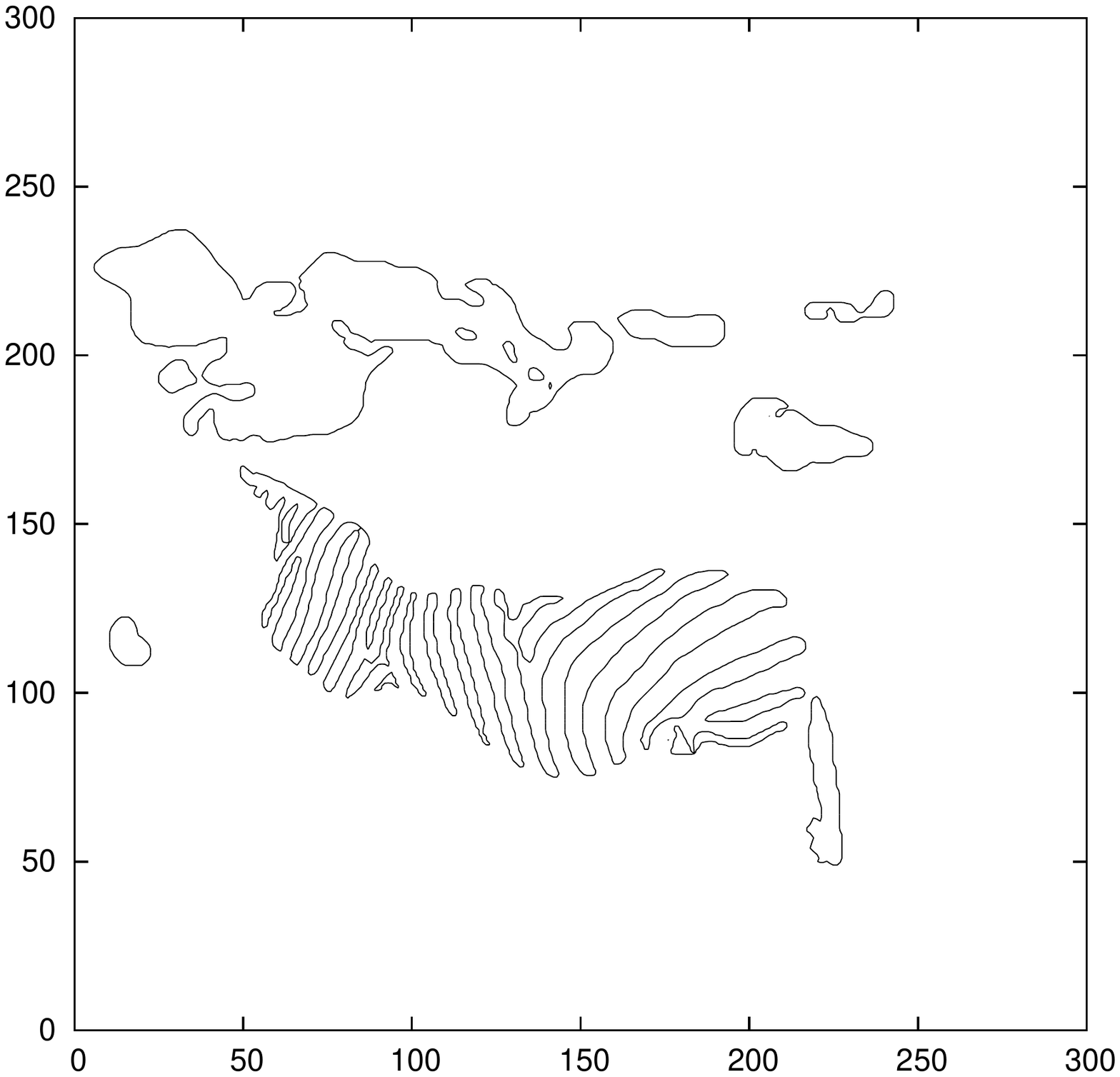}\hspace{-2.5cm}
\ \includegraphics[height=5cm]{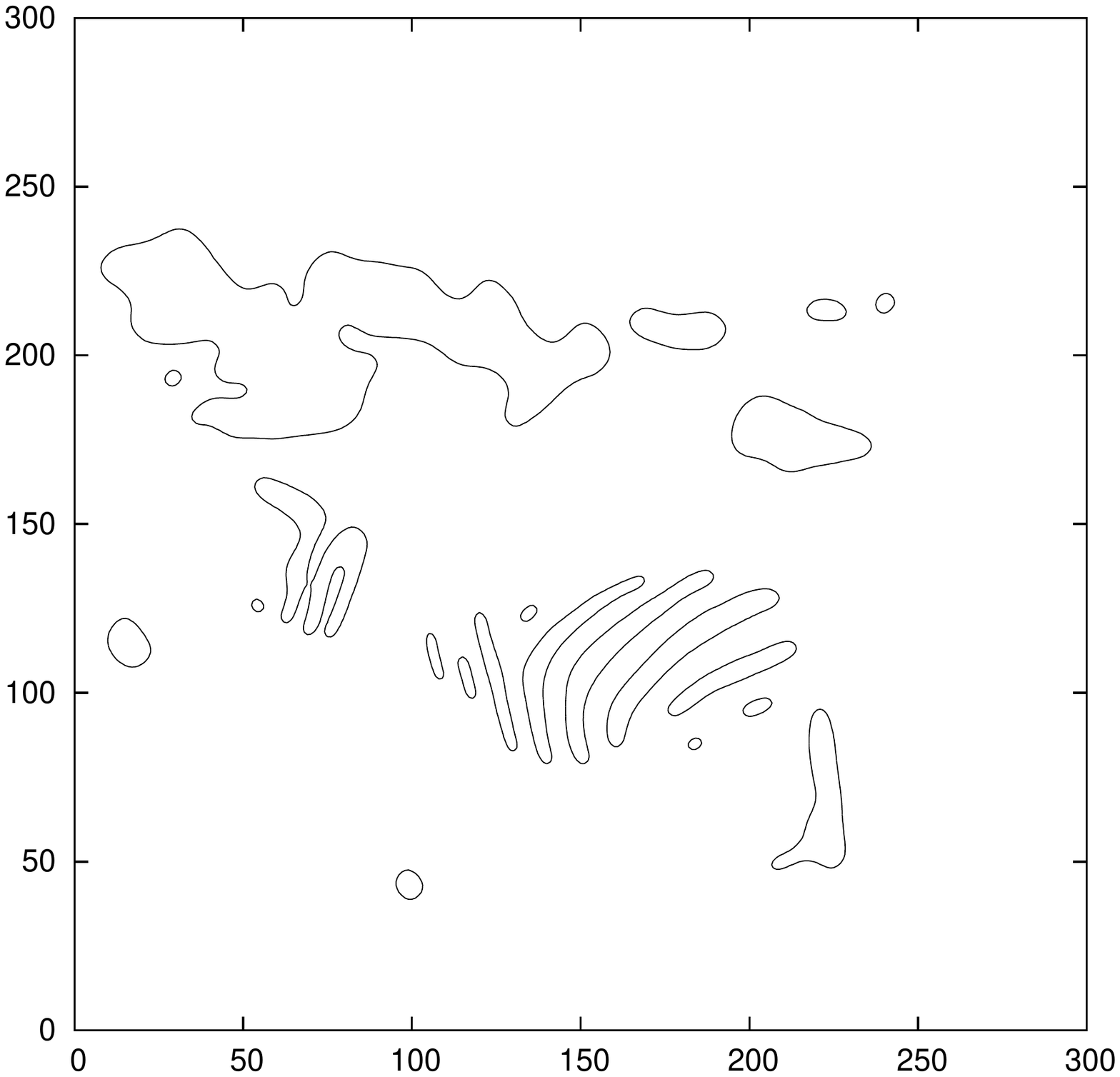}
}
\centerline{
\includegraphics[height=5cm]{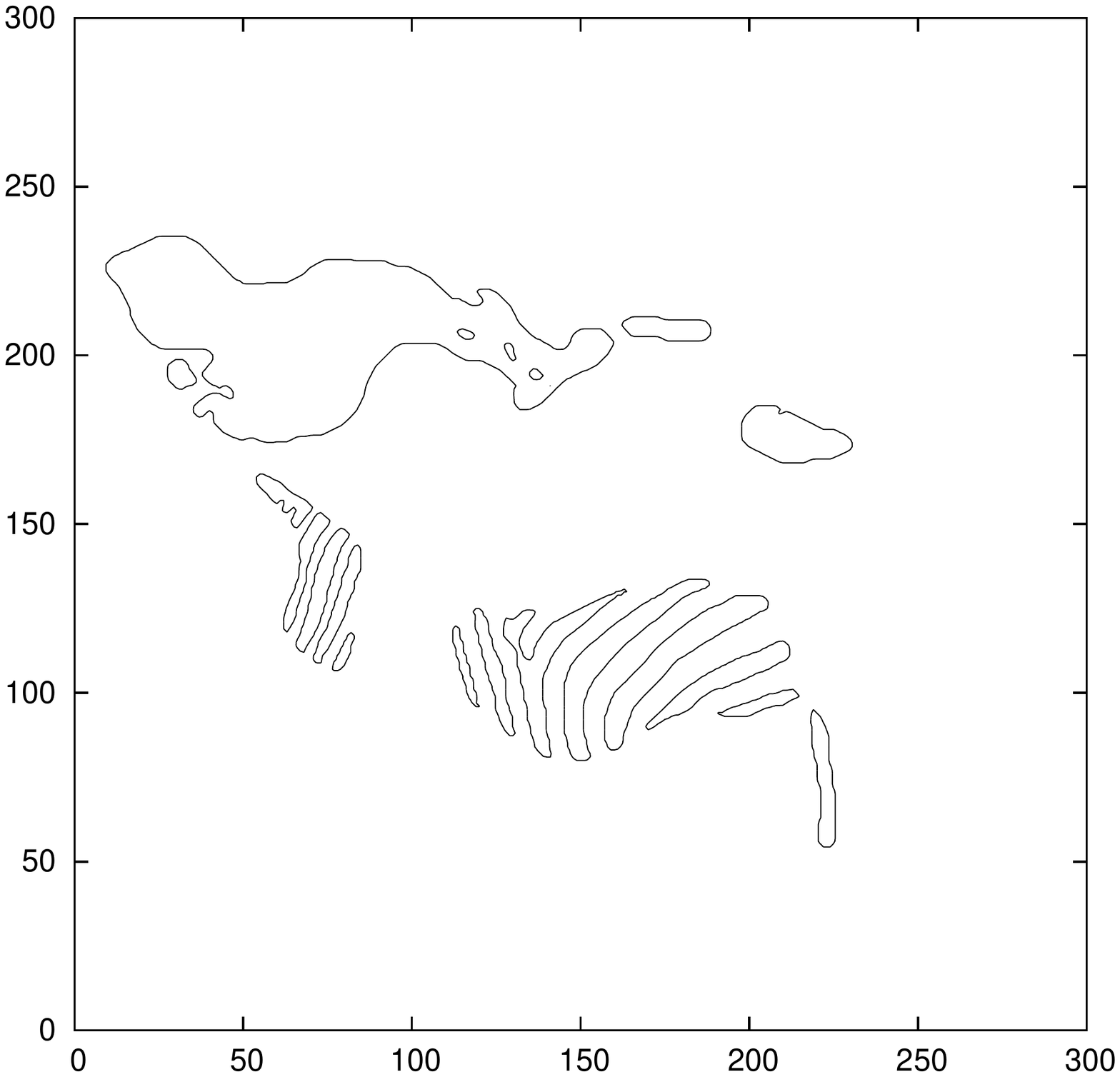}\hspace{-2.5cm}
\ \includegraphics[height=5cm]{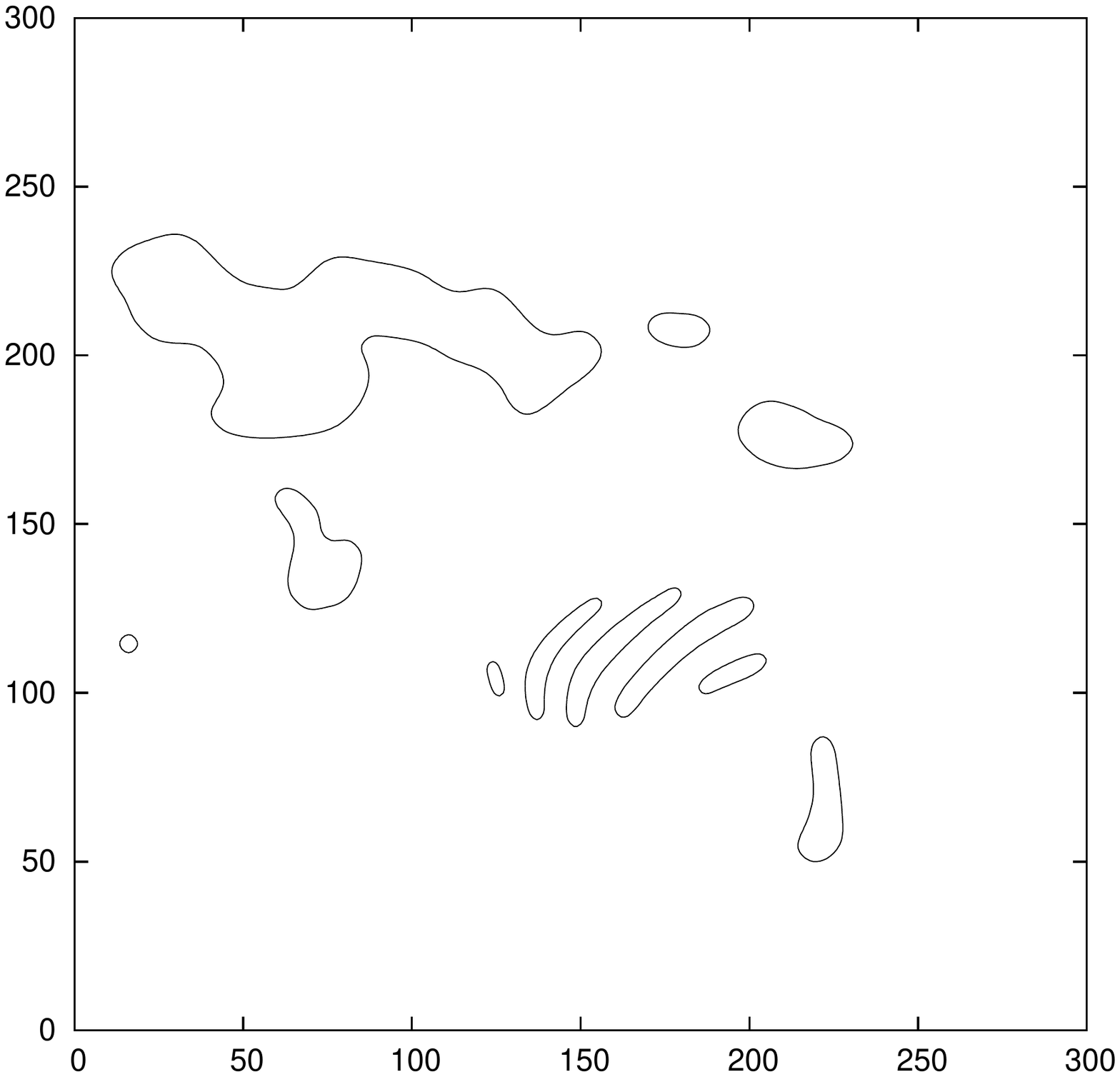}
}\caption{and at later times (left: the non-local motion,
right: the standard curvature flow).}\label{figzebra2}
\end{figure}

It turns out that for this particular energy, there
is an approach, based on a graph representation and the maxflow/mincut
duality, for minimizing binary problems such as
\[
\min_{u_{i,j}\in\{0,1\}} J(u)\,+\,\sum_{i,j} f_{i,j} u_{i,j}
\]
(given any real-valued matrix $(f_{i,j})_{0\le i <N,\,0\le j<M}$), and
an algorithm for minimizing~\eqref{discATW} is easily derived.
See~\cite{CDarbon} for details and in particular~\cite[Appendix~B]{CDarbon}
for how this particular $J$ can be implemented.

\subsection{Examples: two ways to shrink a Zebra}

Figures~\ref{figzebra1}, \ref{figzebra2} and~\ref{figzebra3}
show the motion applied to an initial set of curves with a lot
of oscillations. As expected, the standard curvature motion shrinks the
small scale objects much faster than the one based on the oscillation,
in particular the stripes are preserved longer by the non-local flow.
Notice that it is very difficult to estimate the exact corresponding
times for the two flows, moreover, the numerical imprecision may
provoke sometimes the ``fusion'' of the stripes in the classical curvature
flow (wich is computed also using~\eqref{discATW},
but now $J$ is a discretization of the standard total variation).
\begin{figure}[htb]
\centerline{
\includegraphics[height=5cm]{z20}\hspace{-2.5cm}
\ \includegraphics[height=5cm]{zcm200}
}
\centerline{
\includegraphics[height=5cm]{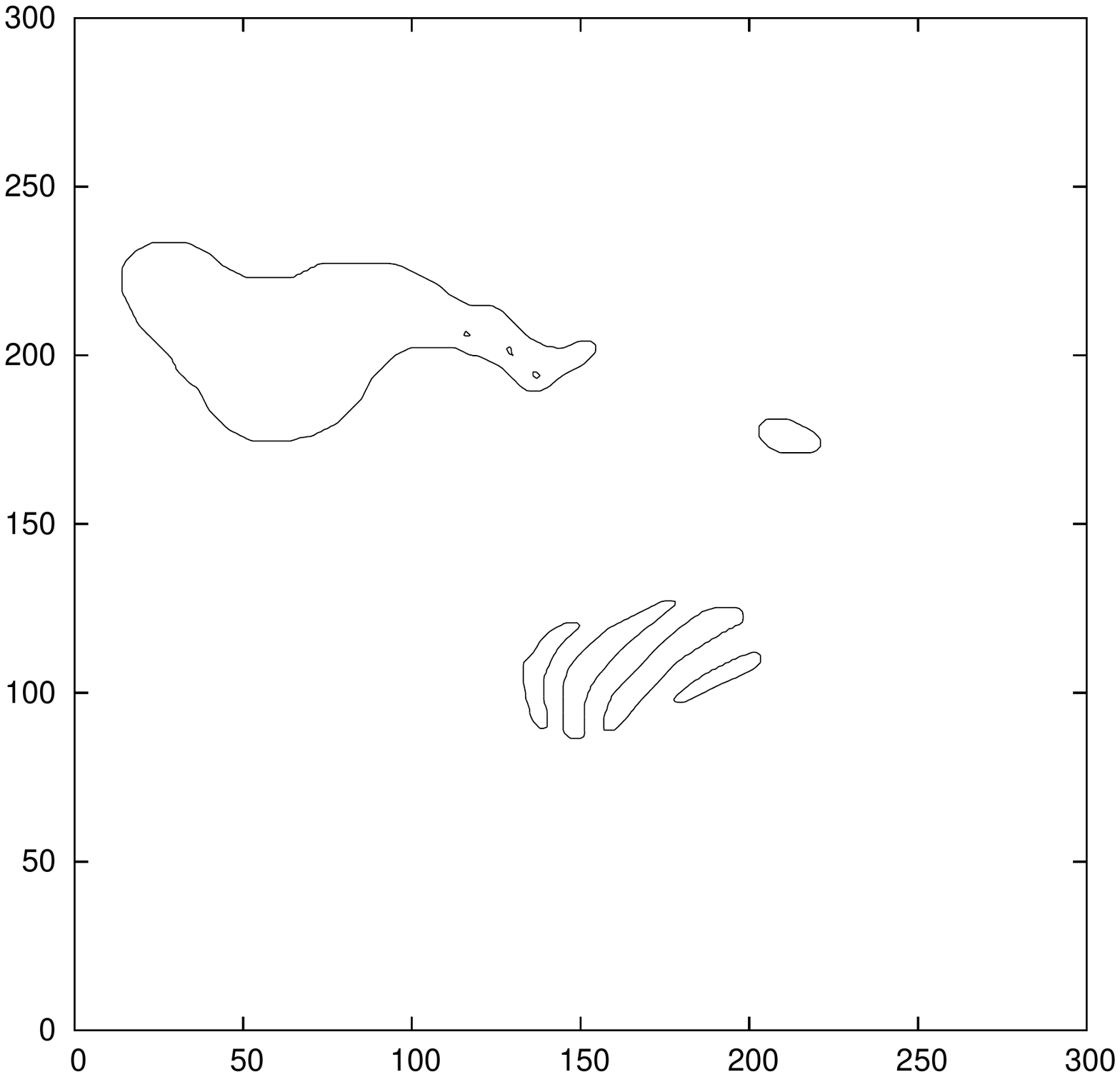}\hspace{-2.5cm}
\ \includegraphics[height=5cm]{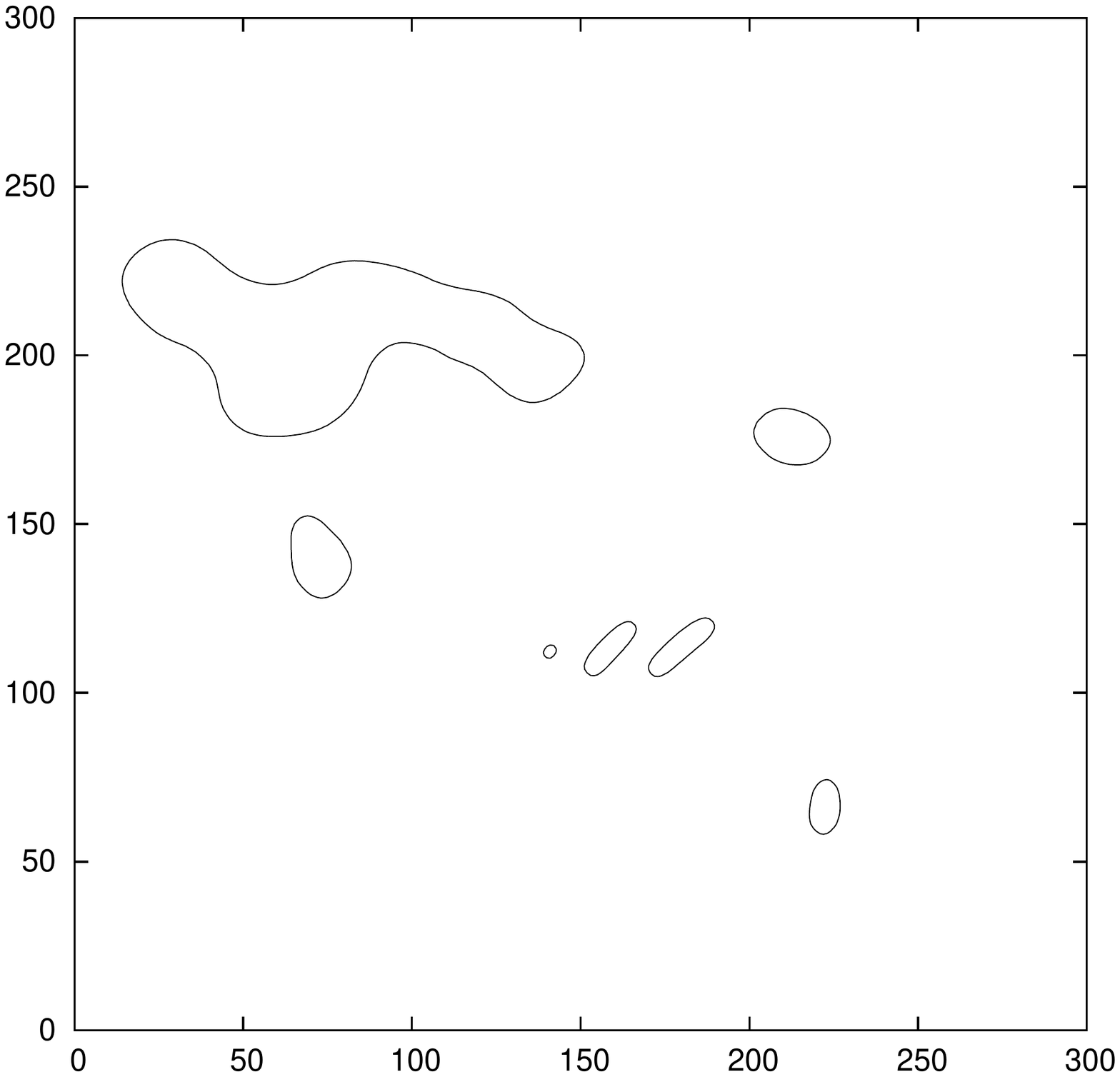}
}\caption{and later...}\label{figzebra3}
\end{figure}

\section*{Acknowledgements} The first author is supported by
the CNRS. Part of this work was done during a  visit at the University
of Parma, funded by GNAMPA and the ERC grant 207573 ``Vectorial Problems".
We wish to thank R.~Monneau, D.~Slep\v{c}ev for helpful
comments. 
We thank the ANR for not having prevented us from doing this
research.
% \problem{WRITE A BIBTEX FILE!}
% \begin{thebibliography}{99}
% \bibitem{BD} Braides, A.; Dal Maso, G.: Non-local approximation of the Mumford-Shah functional. {\it Calc. Var. Partial Differential Equations} {\bf 5} (1997), no. 4, 293--322. 
% \end{thebibliography}

\bibliography{MKCMP}
\end{document}